\documentclass[11pt, reqno]{amsart}  

\usepackage[margin=3 cm]{geometry} 
\usepackage{amsmath,amssymb,amsthm}   
\usepackage{url}       
\usepackage{hyperref}     

\usepackage{xy}
\input xy 
\xyoption{all}

 \usepackage{graphicx} 
\usepackage{pstricks}
\usepackage{epstopdf}

\newcommand{\bF}{{\mathbb F}}
\newcommand{\bN}{{\mathbb N}}

\newcommand{\bQ}{{\mathbb Q}}
\newcommand{\bR}{{\mathbb R}}

\newcommand{\bZ}{{\mathbb Z}}

\newcommand{\fs}{\mathfrak{s}}
\newcommand{\ft}{\mathfrak{t}}
\newcommand{\fC}{\mathfrak{C}}

\newcommand{\mA}{\mathcal{A}}

\newcommand{\mE}{\mathcal{E}}
\newcommand{\mF}{\mathcal{F}}
\newcommand{\mG}{\mathcal{G}}

\newcommand{\mJ}{\mathcal{J}}

\newcommand{\mL}{\mathcal{L}}
\newcommand{\mN}{\mathcal{N}}

\newcommand{\mP}{\mathcal{P}}

\newcommand{\mU}{\mathcal{U}}
\newcommand{\mV}{\mathcal{V}}

\newcommand{\Om}{\Omega}
\newcommand{\om}{\omega}

\newcommand{\Si}{\Sigma}
\newcommand{\la}{\lambda}
\newcommand{\be}{\beta}
\newcommand{\al}{\alpha}
\newcommand{\de}{\delta}

\newcommand{\ga}{\gamma}

\newcommand{\hfp}{HF^+}

\newcommand{\hfhat}{\widehat{HF}}
\newcommand{\hfkhat}{\widehat{HFK}}

\providecommand{\abs}[1]{\lvert#1\rvert}

\newcommand{\bdd}{\partial}

\newcommand{\depth}{\mbox{depth}}

\newtheorem{thm}{Theorem}  [section]
\newtheorem{theorem/definition}{Theorem/Definition}[section]
\newtheorem{prop}[thm]{Proposition}

\newtheorem{cor}[thm]{Corollary}
\newtheorem{lemma}[thm]{Lemma}

\newtheorem*{lemC}{Lemma C}
\newtheorem*{thmA}{Theorem A}
\newtheorem*{thmB}{Theorem B}

\theoremstyle{definition}

\newtheorem{deff}[thm]{Definition}

\theoremstyle{remark}

\newtheorem{rmk}{Remark}

\newtheorem{convention}{Convention}

\DeclareMathOperator{\Hom} {Hom}

\DeclareMathOperator{\Int} {Int}

\DeclareMathOperator{\Spin} {Spin}

\begin{document}

\title{The sutured Floer polytope and taut depth one foliations}         
\date{\today}        
\author{Irida Altman}
\maketitle

\begin{abstract}
For closed 3-manifolds, Heegaard Floer homology is related to the Thurston norm through results due to Ozsv\'ath and Szab\'o, Ni, and Hedden. For example, given a closed 3-manifold $Y$, there is a bijection between vertices of the $HF^+(Y)$ polytope carrying the group $\bZ$ and the faces of the Thurston norm unit ball that correspond to fibrations of $Y$ over the unit circle.  Moreover, the Thurston norm unit ball of $Y$ is dual to the polytope of $\underline{\hfhat}(Y)$.

We prove a similar bijection and duality result for a class of 3-manifolds with boundary called sutured manifolds. A sutured manifold is essentially a cobordism between two surfaces $R_+$ and $R_-$ that have nonempty boundary. We show that there is a bijection between vertices of the sutured Floer polytope carrying the group $\bZ$ and equivalence classes of taut depth one foliations that form the foliation cones of Cantwell and Conlon.    Moreover, we show that a function defined by Juh\'asz, which we call the geometric sutured function, is analogous to the Thurston norm in this context.  In some cases, this function is an asymmetric norm and our duality result is that appropriate faces of this norm's unit ball subtend the foliation cones.

An important step in our work is the following fact: a sutured manifold admits a fibration or a taut depth one foliation whose sole compact leaves are exactly the connected components of $R_+$ and $R_-$, if and only if, there is a surface decomposition of the sutured manifold resulting in a connected product manifold. 

\end{abstract}

\section{Introduction} \label{section intro}

\subsection{Overview}

The inspiration for this paper comes from a connection between two invariants of closed 3-manifolds: the {\it Thurston norm} \cite{Thurston norm} and  {\it Heegaard Floer homology} \cite{OSmain}.  Thurston defined the  norm in the 70's using purely topological terms.  Twenty years later, Ozsv\'ath and Szab\'o developed the Heegaard Floer homology machinery, which draws its roots from complicated pseudo-holomorphic techniques.  In this paper, we aim to show that there is a similar connection between two invariants of a type of 3-manifolds with boundary called {\it sutured manifolds}.  

Let us first consider the closed case. Let $Y$ be a closed oriented 3-manifold.  Recall that the Thurston norm is a seminorm on the homology group $H_2(Y;\bR)$ that measures the minimal `complexity' of surfaces representing an integral homology class. Rational rays through certain faces of the Thurston norm unit ball, called {\it fibred faces}, correspond to fibrations of the  3-manifold over the unit circle (if there are any).  The Heegaard Floer homology invariant $\hfp(Y)$ is a bigraded abelian group, with one of the gradings given by elements of $H^2(Y;\bZ)$.  The support of $\hfp(Y)$ is the set of elements $\fs \in H^2(Y;\bZ)$ for which $\hfp(Y,\fs)$ is nonzero.   The convex hull of the support of $\hfp (Y)$ is a polytope in $H^2(Y;\bR)$.

The connection between the two mentioned invariants is the following: the fibred faces of the Thurston norm ball correspond bijectively to vertices $\fs$ of the $\hfp$ polytope that support $\hfp(Y,\fs)=\bZ$ \cite[Thm.\,1.1]{Ni09}.\footnote{Note that the Thurston (semi)norm of a homology class represented by a torus is zero. Hence if a manifold fibres with fibre a torus, there is no fibred face of the Thurston norm unit ball corresponding to this fibration.  Ni's proof works under the assumption that the fibre has genus greater than one.} This result generalised a theorem of Ghiggini \cite[Thm.\,1.4]{Ghi} that a genus one knot $K$ in $S^3$ is fibered if and only if another appropriate Heegaard Floer group $\hfkhat(K,1)$ is isomorphic to $\bZ$.  In summary, we have the following correspondences:
\begin{equation} \label{closed case}
Y \text{ fibres over } S^1 \overset {Thurston} \Longleftrightarrow \text{fibered face }
\overset {Ni} \Longleftrightarrow \text{ vertex } \fs \text{ with } \hfp(Y,\fs)=\bZ.
\end{equation}

Let us now consider the case of sutured manifolds. A {\it sutured manifold} $(M,\ga)$ is a cobordism $M$ between two surfaces $R_+(\ga)$ and $R_-(\ga)$ that have nonempty boundary, and a set $\ga$ of pairwise disjoint annuli and tori on $\bdd M$ such that $\ga \cup R_+(\ga) \cup R_-(\ga)=\bdd M$. (The surfaces and 3-manifold also have to satisfy certain orientability conditions).  We show that there is a correspondence between taut, depth one foliations of a sutured manifold and the polytope of {\it sutured Floer homology} developed by Juh\'asz \cite{Ju}.  Gabai defined and developed taut, finite depth foliations for sutured manifolds in the 80's.  Twenty years later, Juh\'asz \cite{Ju} extended the `hat flavour' of Heegaard Floer homology for closed 3-manifolds to (balanced) sutured manifolds.  But before we can give a statement analogous to \eqref{closed case}, we introduce the key ingredients. 

We call $\mF$ a foliation of $(M,\ga)$ if the leaves of $\mF$ are transverse to $\ga$ and tangential to $R(\ga):=\bdd M \setminus \Int(\ga)$.  For example, if $M$ is a solid torus and $\ga=\bdd M$, then the fibration of $M=S^1\times D^2 \to S^1$ given by the projection onto the first coordinate, is a depth zero foliation of $(M,\ga)$.  We only work with foliations that are smooth and transversely oriented. 

Cantwell and Conlon showed that taut depth one foliations form open, convex, polyhedral cones in $H^1(M;\bR)=H_2(M,\bdd M;\bR)$ called {\it foliation cones} \cite{CC99}.  Note that rational rays through the fibred faces of the Thurston norm unit ball also live in open, convex, polyhedral cones (which we could call {\it fibred cones}) that are subtended by the fibred faces.  So similarly to the first correspondence in \eqref{closed case}, we have
\begin{equation*}
(M,\ga) \text{ has taut depth 1 foliation } \overset {Cantwell} {\underset{Conlon} \Longleftrightarrow} \text{foliated face. }
\end{equation*}

On the other hand, sutured Floer homology associates to a balanced sutured manifold $(M,\ga)$ a finitely generated abelian group denoted by $SFH(M,\ga)$.  One of the gradings is given by a set of so called relative $\Spin^c$ structures of $(M,\ga)$, which can be identified with $H^2(M,\bdd M;\bZ)$.   Then the support of $SFH(M,\ga)$ is defined to be the set of elements $\fs \in H^2(M,\bdd M;\bZ)$ for which $SFH(M,\ga,\fs)$ is nonzero.  The convex hull of the support is the {\it sutured Floer polytope}  $P(M,\ga)$ living in $H^2(M, \bdd M;\bR)$  \cite{Jupolytope}.   

Our connection between the two theories, given in Theorem A, is that the foliation cones of $(M,\ga)$ correspond bijectively to vertices $\fs$ of $P(M,\ga)$ that support $SFH(M,\ga,\fs)=\bZ$.  Thus, with Theorem A, we have the following set of correspondences those in \eqref{closed case}:
\begin{equation} \label{open case}
\begin{split}
(M,\ga) \text{ has taut depth 1 foliation } &\overset {Cantwell} {\underset{Conlon} \Longleftrightarrow} \text{foliated face }\\
&\overset {Thm A} 
\Longleftrightarrow \text{ vertex } \fs \text{ with } SFH(M,\ga,\fs)=\bZ.
\end{split}
\end{equation}

The relationships between the old and new theories do not stop there.  In particular, Ozsv\'ath and Szab\'o showed that the Thurston norm unit ball is dual  to the polytope of $\underline{\hfhat}(Y)$, another flavour of Heegaard Floer homology with so called twisted coefficients  \cite[Thm.\,1.1]{OS genus}.  In Theorem B, we show that the foliation cones of Cantwell and Conlon are subtended by faces of the dual sutured Floer polytope.  Note that the foliation cones do not form a polytope, and this agrees with the fact that the sutured Floer polytope is defined up to translation in $H^2(M,\bdd M; \bR)$.  In the closed case there is no such ambiguity, as both the Thurston norm and and $\underline{\hfhat}(Y)$ are well-defined, that is, `centred' around the origin of their ambient vector space.

It may be helpful to bear in mind the following dictionary of terms between the closed case and the case of balanced sutured manifolds, especially if one is already familiar with the former.
\vspace{0.05cm}
\begin{center}
  \begin{tabular}{  c | c  }
  
    $Y$ closed 3-manifold & $(M,\ga)$ balanced sutured manifold \\ \hline
    Thurston norm $x$ &  geometric sutured function $y_t$  \\ 
    $Y$ fibres over $S^1$ & $(M,\ga)$ admits a taut depth one foliation \\
    fibered face of $x$ unit norm ball & foliated face of $y_t$ `unit norm ball' \\
    vertex $\fs$ with $\hfp(Y,\fs)=\bZ$ & vertex $\fs$ with $SFH(M,\ga,\fs)=\bZ$ \\
    $\underline{\hfhat}(Y)$ polytope & $SFH(M,\ga)$ polytope
  \end{tabular}
\\
\end{center}

\vspace{0.5cm}

 We have not yet talked about the {\it geometric sutured function $y_t$}, which is a map $H_2(M,\bdd M;\bR) \to \bR$, but it is easiest if we delay its introduction until Subsection \ref{subsection norms}.  The function $y_t$ can take negative values, hence it may not have a unit norm ball per se.  However, when $y_t$ is a seminorm, then its unit ball is dual to the sutured polytope $P(M,\ga)$; see Corollary \ref{yt dual to polytope} for details.   

\subsection{Statements of results}

The polytope $P(M,\ga)$ is said to have an {\it extremal $\bZ$ at $\fs$}, if $\fs$ is a vertex of the polytope  and $SFH(M,\ga,\fs)=\bZ$.  Saying that $\fs$ is {\it extremal with respect to a homology class $\al \in H_2(M,\bdd M;\bR)$} means that $\al(\fs) > \al(\ft)$ for any other vertex $\ft$ of the polytope.  Set $R(\ga):=R_+(\ga) \cup R_-(\ga)$ and $M_0:=M \setminus R(\ga)$.  If $\mF$ is a depth one foliation, then the manifold $M_0$ fibres over $S^1$ and this fibration defines an element $\la(\mF) \in H^1(M;\bR)$ (Lemma \ref{lemma fibring}).  Moreover, $\la(\mF)$ gives rise to a `foliation ray of $\mF$' in $H^1(M; \bR)=H_2(M,\bdd M;\bR)$ given by $r \cdot \left( PD \cdot \la(\mF)\right)$ for $r \in \bR^{\geq 0}$, and this foliation ray is contained in a foliation cone of Cantwell and Conlon.  Here we denote by $PD$ the Poincar\'e-Lefschetz duality map.

We can now state the first of our two main theorems.

 \begin{thmA}
Suppose $(M,\ga)$ is a strongly balanced sutured manifold with $H_2(M;\bZ)=0$, and let $P(M,\ga)$ denote its sutured polytope. Then $P(M,\ga)$ has an extremal $\bZ$ at a $\Spin^c$ structure $\fs$ if and only if there exists a taut depth one foliation $\mF$ of $(M,\ga)$ whose sole compact leaves are the connected components of $R(\ga)$ and such that $\fs$ is extremal with respect to $PD \circ \la(\mF)$.
\end{thmA}

The condition that $H_2(M; \bZ)=0$ is purely technical; it carries over from theorems about how sutured Floer homology behaves under surface decompositions.  An idea for removing this condition has been communicated to the author by S. Friedl \cite{F13}.

As we mentioned above, the Thurston norm unit ball is dual to the polytope of $\underline{\hfhat}(Y)$.  Since the Thurston norm is centrally symmetric, this duality result recovers the symmetries of $\hfhat$. However, the sutured Floer polytope is not symmetric in general; for examples see Subsection \ref{subsection examples} below and \cite[Sec.\,8]{FJR10}.  So there is no hope that $P(M,\ga)$ is the dual of a  (semi)norm unit ball. Moreover, there is no canonical identification of $\Spin^c(M,\ga)$ with $H^2(M,\bdd M;\bZ)$, so $P(M,\ga)$ is defined in $H^2(M,\bdd M;\bR)$ only up to translation.   

Nevertheless, under certain conditions (for example when $H_2(M; \bZ)=0$) we can define a {\it geometric sutured function} $y_t$ that is analogous to the Thurston norm.  This function $y_t$ is a sum of two symmetric terms and an asymmetric term that reflects the choice of identification $\Spin^c(M,\ga) \to H^2(M,\bdd M; \bZ)$ (see Subsection \ref{subsection norms}).  For those already familiar with Juh\'asz's work, $y_t$ is defined using the function $c(S,t)$ that depends on the topology of a properly embedded surface $S$ in $M$, and on the trivialisation $t$ of a particular plane bundle (see \cite[Def.\,3.16]{Jupolytope} or \eqref{equation cst} below).   However, we advertise the geometric nature of $P(M,\ga)$ explicitly by putting it side by side with other seminorms for 3-manifolds.

The preceding discussion implies that there is no obvious definition of a dual sutured polytope.  Even so, we can always obtain the {\it dual sutured cones} $Q(M,\ga)$ of the polytope.  In general, we can define the dual cones $Q$ of any polytope $P$. Let $P$ be given by vertices $v_1, \ldots, v_{n}$ living in a vector space $V$ over some field $\bF$.  The set $Q$ is a collection of polyhedral cones $Q_1, \ldots, Q_n$ in the dual space $V^*= \Hom(V, \bF)$ where
\[
Q_i:=\{v^* \in V^* : v^*(v_i) > v^*(v_j) \textrm{ for } i \neq j\}.
\]
In particular, the cones of $Q(M,\ga)$ correspond to vertices of $P(M,\ga)$.  The cones that correspond to extremal $\bZ$ vertices of $P(M,\ga)$ are referred to as the {\it extremal $\bZ$ cones} and denoted by $Q_\bZ(M,\ga)$.

Finally, we state the second main theorem.

\begin{thmB} 
Let $(M,\ga)$ be a taut, strongly balanced sutured manifold with $H_2(M)=0$.  The extremal $\bZ$ cones $Q_\bZ(M,\ga)$ are precisely the foliation cones defined by Cantwell and Conlon in \cite{CC99}.
\end{thmB}

Most interesting sutured manifolds are strongly balanced, so this is not a significant restriction in Theorem B.  For a precise definition of strongly balanced see Definition \ref{deff strongly balanced}. 

Theorem B provides a connection between two areas of low-dimensional topology that use seemingly different tools.  The construction of sutured Floer homology relies on pseudoholomorphic techniques which form the basis of Heegaard Floer homology, whereas Cantwell and Conlon prove the existence of foliation cones by using foliation currents \cite{Sullivan,Schwartzmann}.

Lastly, we would like to mention a key step in our work stated in Lemma C.  In the case of a closed 3-manifold $Y$ that fibres over $S^1$, cutting $Y$ along a fibre yields a product manifold homeomorphic to $\text{fibre}\times I$.  Conversely, given an automorphism of a closed surface $\varphi \colon S \to S$, the mapping torus fibres over $S^1$ with fibre $S$.  In the case of a sutured manifold $(M,\ga)$, Lemma C states that cutting $(M,\ga)$ along a properly embedded surface $S$ gives a connected product manifold if and only if $S$ is either a fibre of a fibration $M\to S^1$ or can be made into a leaf of a depth one foliation of $(M,\ga)$ by an operation called 'spinning'.  We give the precise definitions of all the terminology used in Lemma C at appropriate places throughout Sections \ref{section sutured mflds} and \ref{section foliations}.

\begin{lemC} 
Suppose $(M,\gamma)$ is a connected sutured manifold.  Let $(M,\ga)\leadsto^S (M',\ga')$ be a surface decomposition along $S$ such that $(M',\ga')$ is taut.  Then $(M',\gamma')$ is a connected product sutured manifold
if and only if either
\begin{enumerate}
\item $R(\ga)=\emptyset$ and $S$ is the fibre of a depth zero foliation $\mF$ given by a fibration $\pi \colon M \to S^1$, \\
or
\item $R(\ga) \neq \emptyset$ and $S$ can be spun along $R(\ga)$ to be a leaf of a depth one foliation $\mF$ of $(M,\ga)$ whose sole compact leaves are the connected components of $R(\ga)$.
\end{enumerate}
Up to equivalence, all depth zero foliations of $(M,\ga)$, and all depth one foliations of $(M,\ga)$ whose sole compact leaves are the connected components of $R(\ga)$, are obtained from a surface decomposition resulting in a connected product sutured manifold.
\end{lemC}

\begin{rmk}
The gist of Lemma C is known to experts \cite{Conlon private, Gabai private}, but the author was unable to find any written references.  Cantwell and Conlon are preparing a paper exploring the relationship of sutured manifold decomposition and foliations from the perspective of staircases and junctures \cite{CCSmooth} that will include a proof of Lemma C.  Below we give our own proof of Lemma C, together with the minimal necessary background from foliation theory.
\end{rmk}

\subsection{Organisation of the article} The article is organised as follows.  In Section \ref{section sutured mflds}, we give the background on the sutured Floer polytope and surface decompositions, followed by a survey of  seminorms on the homology of 3-manifolds.    In Section \ref{section foliations}, we describe Gabai's construction of depth one foliations, followed by the theory of junctures and spiral staircases of foliations, and an introduction to Cantwell and Conlon's foliation cones.  Finally, in Section \ref{section duality} we prove Lemma C, Theorem A and Theorem B (in that order), and we present a few examples that illustrate the duality between foliation cones and the sutured Floer polytope.  Notational conventions are explained at the beginning of Section \ref{section sutured mflds}.

\subsection{Acknowledgements} I would like to thank my Ph.D. advisers Stefan Friedl and Andr\'as Juh\'asz for suggesting this project to me, as well as for many helpful discussions and suggestions.  Many thanks to Lawrence Conlon for explaining aspects of depth one foliation theory, in particular, the concepts of spiral staircase and junctures, and for his insightful correspondence on foliation cones. I thank my Ph.D. adviser Saul Schleimer for his interest in my work.

I owe much gratitude to the University of Warwick and the Warwick Mathematics Institute for generously supporting me through a Warwick Postgraduate Research Scholarship, as well as to the Department for Pure Mathematics and Mathematical Statistics in Cambridge for their hospitality.

\section{Sutured manifolds} \label{section sutured mflds}

This section gives the background to understanding the sutured Floer homology side of the duality statement.  We begin with the basic definitions of sutured manifolds and the operation of surface decomposition \cite{Gabai}.  This is followed by a brief review of sutured Floer homology and relative $\Spin^c$ structures \cite{Ju}, as well as the definition of the sutured Floer polytope \cite{Jupolytope}.  Then there is a description of how sutured Floer homology behaves under  decompositions along ``well-behaved'' surfaces  \cite{Jusurface}. The section ends with a survey of seminorms for 3-manifolds: the Thurston norm \cite{Thurston norm}, the generalised Thurston norm \cite{Scharlemann}, the sutured Thurston norm \cite{CC sutured Thurston norm},  the sutured seminorm \cite{Jupolytope}, and finally what we call the geometric sutured function.

\subsection{Notation}    If two topological spaces $W$ and $X$ are homeomorphic, we write $W \cong X$.  If $U$ is an open set in $X$, then $\overline U$ denotes the closure of $U$ in the topology of $X$.    Denote by $\abs X$ the number of connected components of $X$.

All homology groups are assumed to be given with $\bZ$ coefficients unless otherwise stated.  Let $M$ be a $n$-manifold with boundary.  Then $PD$ denotes the Poincar\'e-Lefschetz duality map $H_*(M,\bdd M)  \to H^{n-*}(M)$.  As this map is an isomorphism, we simplify the notation and also call the inverse map $PD$; it is obvious from the context to which map we are referring. 

Let $M$ be an $n$-manifold, and $L \subset M$ a codimension-1 submanifold.  Then a tubular neighbourhood of $L$, denoted by $N(L)$, is often parametrized as $L \times (0,1)$.  We write $L \times s$ to mean $L \times \{s\}:=\{(x,s) \in M: x \in L\}$ for some $s \in [0,1]$.  Similarly, in general, when $J \times (a,b) \times (c,d)$ is the parametrisation of a tubular neighbourhood of a codimension-2 submanifold $J \subset M$, we write $J \times s \times t$ or $(J,s,t)$ to mean the codimension-2 manifold $J \times \{ s \} \times \{t\}$, where $(s,t) \in [a,b] \times [c,d]$.

\subsection{Sutured manifolds} \label{subsection sutured manifolds}

Sutured manifolds were defined by Gabai \cite[Def.\,2.6]{Gabai}.

\begin{deff}
A {\it sutured manifold} $(M, \gamma)$ is a compact oriented 3-manifold $M$ with boundary together with a set $\gamma \subset \bdd M$ of pairwise disjoint annuli $A(\gamma)$ and tori $T(\ga)$. Furthermore, in the interior of each component of $A(\ga)$ one fixes a suture, that is, a homologically nontrivial oriented simple closed curve. We denote the union of the sutures by $s(\ga)$.

Finally, every component of $R(\gamma): = \bdd M \setminus \Int(\ga)$ is oriented. Define $R_+(\ga)$ (or $R_-(\ga)$) to be those components of $\bdd M \setminus \Int(\ga)$ whose normal vectors point out of (into) $M$. The orientation on $R(\ga)$ must be coherent with respect to $s(\ga)$, that is, if $\de$ is a component of $\bdd R(\ga)$ and is given the boundary orientation, then $\de$ must represent the same homology class in $H_1(\ga)$ as some suture.
\end{deff}

Sutured Floer homology is defined on a wide subclass of sutured manifolds called balanced sutured manifolds, whereas the sutured Floer polytope is defined for strongly balanced sutured manifolds.

\begin{deff} \cite[Def.\,2.2]{Ju}
A {\it balanced} sutured manifold is a sutured manifold $(M, \ga)$ such that $M$ has no closed components, the equality $\chi(R_+(\ga)) = \chi(R_-(\ga))$ holds, and the map $\pi_0(A(\ga)) \to \pi_0(\bdd M)$ is surjective.
\end{deff}

\begin{deff} \cite[Def.\,3.5]{Jusurface} \label{deff strongly balanced}
A {\it strongly balanced} sutured manifold is a balanced sutured manifold $(M, \ga)$ such that for every component $F$ of $\bdd M$ the equality, $\chi(F \cap R_+(\ga)) = \chi(F \cap R_-(\ga))$ holds.
\end{deff}

The most frequently studied and most ``interesting'' examples of sutured manifolds are all strongly balanced.  A trivial example is the {\it product sutured manifold} given by $(\Si \times I, \bdd \Si \times I)$ where $\Si$ is a surface with boundary and with no closed components.  

\begin{convention}
For the purposes of this paper a product sutured manifold is always assumed to be connected.
\end{convention}

 Other simple examples are obtained from any closed, connected 3-manifold by removing a finite number of 3-balls and adding one trivial suture to each spherical boundary component.  Less trivial examples are those of link complements in closed 3-manifolds with sutures consisting of an even number of $(p,q)$-curves on the toroidal components, as well as, the complements of surfaces in closed 3-manifolds, endowed with sutures derived from the boundary of the surface (e.g. the complement of a Seifert surface of a knot)\footnote{More precisely, one does not take the complement of surface $\Si$ in a closed 3-manifold $Y$, but instead the complement of a double collar neighbourhood $\Si \times (-1,1)$ of $\Si$.  Then the sutures are the curves corresponding to $\bdd S \times \{0\}$ on $\bdd \left(Y \setminus \Si \times (-1,1) \right)$.}.

\begin{deff}
A sutured manifold $(M,\ga)$ is said to be {\it taut} if $M$ is irreducible, and $R(\ga)$ is incompressible and Thurston-norm minimising in $H_2(M,\ga)$.
\end{deff}

Lastly, we define the operation of decomposing sutured manifolds into simpler pieces which was introduced by Gabai \cite[Def.\,3.1]{Gabai}.  

\begin{deff} \label{def decomp}
Let $(M,\ga)$ be a sutured manifold and $S$ a properly embedded surface in $M$ such that for every component $\la$ of $S \cap \ga$ one of (i)--(iii) holds:
\begin{enumerate}
\item $\la$ is a properly embedded non-separating arc in $\ga$.
\item $\la$ is  simple closed curve in an annular component $A$ of $\ga$ in the same homology class as $A \cap s(\ga)$.
\item $\la$ is a homotopically nontrivial curve in a toroidal component $T$ of $\ga$, and if $\de$ is another component of $T \cap S$, then $\la$ and $\de$ represent the same homology class in $H_1(T)$.
\end{enumerate}
Then $S$ defines a {\it sutured manifold decomposition} 
\[
(M,\ga) \leadsto^S (M',\ga'),
\]
where $M':=M \setminus \Int(N(S))$ and 
\begin{align*}
\ga': & = (\ga \cap M') \cup N(S_+' \cap R_-(\ga)) \cap N(S'_- \cap R_+(\ga)),\\
R_+(\ga'): & = ((R_+(\ga) \cap M') \cup S_+') \setminus \Int(\ga'), \\
R_-(\ga'): & = ((R_-(\ga) \cap M') \cup S_-') \setminus \Int(\ga'),
\end{align*}
where $S_+'$ ($S'_-$) are the components of $\bdd N(S) \cap M'$ whose normal vector points out of (into) $M$.
\end{deff}

The manifolds $S_+$ and $S_-$ are defined in the obvious way as copies of $S$ embedded in $\bdd M'$ that are obtained by cutting $M$ along $S$.

\subsection{Sutured Floer homology and the sutured Floer polytope} \label{subsection polytope}

First of all, here is some background on $\Spin^c$ structures.  The following definition of relative $\Spin^c$ structures originates from Turaev's work \cite{Tu90}, but in the current phrasing comes from \cite{Ju}.  For proofs and details we refer to Juh\'asz's papers \cite{Ju,Jusurface} and \cite{Jupolytope}.

Fix a Riemannian metric on $(M,\ga)$.  Let $v_0$ denote a nonsingular vector field on $\bdd M$ that points  into $M$ on $R_-(\ga)$ and out of $M$ on $R_+(\ga)$, and that is equal to the gradient of a height function $s(\ga) \times I \to I$ on $\ga$. The space of such vector fields is contractible.

A relative $\Spin^c$ structure is defined to be a {\it homology class} of vector fields $v$ on $M$ such that $v|{\bdd M}$ is equal to $v_0$.  Here two vector fields $v$ and $w$ are said to be {\it homologous} if there exists an open ball $B \subset \Int(M)$ such that $v$ and $w$ are homotopic through nonsingular vector fields on $M \setminus B$ relative to the boundary.  There is a free and transitive action of $H_1(M)=H^2(M,\bdd M)$ on $\Spin^c(M,\ga)$ given by {\it Reeb turbularization} \cite[p.\,639]{Tu90}.  This action makes the set $\Spin^c(M,\ga)$ into an $H_1(M)$-torsor.  From now on, we refer to a map $\iota \colon \Spin^c(M,\ga) \to H_1(M)$ as an  {\it identification} of the two sets if $\iota$ is an $H_1(M)$-equivariant bijection.  Note that $\iota$ is completely defined by which element $\fs \in \Spin^c(M,\ga)$ it sends to $0 \in H_1(M)$ (or any other fixed element of $H_1(M)$).

The perpendicular two-plane field $v_0^\perp$ is trivial on $\bdd M$ if and only if $(M,\ga)$ is strongly balanced \cite[Prop.\,3.4]{Jusurface}.  Suppose that $(M,\ga)$ is strongly balanced.  Define $T(M,\ga)$ to be the set of trivialisations of $v_0^\perp$.  Let $t \in T(M,\ga)$.  Then there is a map dependent on the choice of trivialisation,
\[
c_1(\cdot, t) \colon \Spin^c(M,\ga) \to H^2(M,\bdd M),
\]
where $c_1(\fs,t)$ is defined to be the relative Euler class of the vector bundle $v^\perp \to M$ with respect to a partial section coming from a trivialisation $t$.  So $c_1(\fs,t)$ is the first obstruction to extending the trivialisation $t$ of $v_0^\perp$ to a trivialisation of $v^\perp$.  Here $v$ is a vector field on $M$ representing the homology class $\fs$.

Sutured Floer homology associates to a given balanced sutured manifold $(M,\ga)$ a finitely generated bigraded abelian group denoted by $SFH(M,\ga)$.  The group $SFH(M,\ga)$ is graded by the relative $\Spin^c$ structures $\fs \in \Spin^c(M,\ga)$, and has a relative $\bZ_2$ grading. In particular, for each $\fs \in \Spin^c(M,\ga)$ there is a well-defined abelian group $SFH(M,\ga,\fs)$ \cite{Ju}, and the direct sum of these groups forms the {\it sutured Floer homology} of $(M,\ga)$.  That is,
\[
SFH(M,\ga):=\bigoplus_{\fs \in \Spin^c(M,\ga)} SFH(M,\ga,\fs).
\]

A very important property of sutured Floer homology is that it detects the product sutured manifold.  In the following theorem we put together \cite[Prop.\,9.4]{Ju} and \cite[Thm.\,9.7]{Jusurface}.

\begin{thm} \label{thm product}
An irreducible balanced sutured manifold $(M,\ga)$ is a product manifold if and only if $SFH(M,\ga)=\bZ$.
\end{thm}

Note that the statement would be false without the word irreducible: if  $P(1)$ is the Poincar\'e homology sphere with a 3-ball removed and a single suture along the the spherical boundary, then $SFH(P(1))=\bZ$ \cite[Rmk.\,9.5]{Ju}, but $P(1)$ is not a product (and not irreducible by definition).  

We now have all the ingredients required to define the sutured Floer polytope.  Let $S(M,\ga)$ be the {\it support} of the sutured Floer homology of $(M,\ga)$.  That is,
\[
S(M,\ga):=\{ \fs \in \Spin^c(M,\ga) \colon SFH(M,\ga,\fs)\neq 0\}.
\]
Consider the map $i \colon H^2(M,\bdd M;\bZ) \to H^2(M,\bdd M;\bR)$ induced by the inclusion $\bZ \hookrightarrow \bR$.  For $t$ a trivialisation of $v_0^\perp$, define
\[
C(M,\ga,t):=\{ i \circ c_1(\fs,t) : \fs \in S(M,\ga)\} \subset H^2(M,\bdd M;\bR).
\]
\begin{deff}
The {\it sutured Floer polytope} $P(M,\ga,t)$ with respect to $t$ is defined to be the convex hull of $C(M,\ga,t)$.
\end{deff}

Next, we have that $c_1(\fs,t_1)-c_1(\fs,t_2)$ is an element of $H^2(M,\bdd M)$ dependent only on the trivialisations $t_1$ and $t_2$ \cite[Lem.\,3.11]{Jupolytope}, and therefore we may write $P(M,\ga)$ to mean the polytope in $H^2(M,\bdd M;\bR)$ up to translation. 

 It is important to note that $c_1$ ``doubles the distances.''  Namely, for a fixed trivialisation $t$ and $\fs_1,\fs_2 \in \Spin^c(M,\ga)$, Lemma 3.13 of \cite{Jupolytope} says that
 \[
 c_1(\fs_1,t) - c_1(\fs_2,t)=2(\fs_1- \fs_2),
 \]  
 where $\fs_1-\fs_2$ is the unique element $h \in H^2(M,\bdd M)$ such that $\fs_1=h + \fs_2$.  Such an element exists and is unique by definition of a $H^2(M,\bdd M)$-torsor.
 
Let $t \in T(M,\ga)$.  Then an element $\al \in H_2(M,\bdd M;\bR)$ defines subsets $P_\al(M,\ga,t)$ and $C_\al(M,\ga,t)$ of $P(M,\ga,t)$ and $C(M,\ga,t)$, respectively \cite[p.17]{Jupolytope}.  Firstly, set 
\begin{equation} \label{equation calt}
c(\al,t):=\min \{ \langle c, \al \rangle : c \in P(M,\ga,t) \}.
\end{equation}
Then there is a subset $H_\al \subset H^2(M,\bdd M;\bR)$ given by
\[
H_\al :=\{ x \in H^2(M,\bdd M;\bR) : \langle x, \al \rangle=c(\al,t)\}.
\]
Lastly,
\begin{gather} 
P_\al(M,\ga,t):=H_\al \cap P(M,\ga,t), \text{ and }
C_\al(M,\ga,t):=H_\al \cap C(M,\ga,t), \\ \label{equation SFH alpha}
SFH_\al(M,\ga):=\bigoplus \{ SFH(M,\ga,\fs) : i(c_1(\fs,t)) \in C_\al(M,\ga,t)\}. 
\end{gather}

For an explanation of the types of well-behaved surfaces mentioned in the last part of this subsection see Definitions \ref{def groomed} and \ref{def nice}.

If $(M,\ga)$ is taut and strongly balanced, then $P_\al(M,\ga,t)$ is the convex hull of $C_\al(M,\ga,t)$ and it is a face of the polytope $P(M,\ga,t)$.  Furthermore, if $S$ is a nice decomposing surface that gives a taut decomposition $(M,\ga) \leadsto^S (M',\ga')$ and $[S]=\al$, then $SFH(M',\ga')=SFH_\al(M,\ga)$ \cite[Prop.\,4.12]{Jupolytope}. 

We conclude this subsection with the following theorem.

\begin{thm} \cite[Cor.\,4.15]{Jupolytope} \label{thm decomposition}
Let $(M,\ga)$ be a taut balanced sutured manifold, and suppose that $H_2(M)=0$.  Then the following two statements hold.
\begin{enumerate}
\item For every $\al \in H_2(M,\bdd M)$, there exists a groomed surface decomposition\linebreak $(M,\ga) \leadsto^S (M',\ga')$ such that $(M',\ga')$ is taut, $[S]=\al$, and 
\[
SFH(M',\ga') \cong SFH_\al(M,\ga).
\]
If moreover, $\al$ is well-groomed, then $S$ can be chosen to be well-groomed.
\item For every face $F$ of $P(M,\ga,t)$, there exists an $\al \in H_2(M,\bdd M)$ such that $F=P_\al(M,\ga,t)$.
\end{enumerate}
\end{thm}

\subsection{Well-behaved surfaces} \label{subsection well-behaved surfaces}
The result of decomposition along some surfaces can be described more easily than along others.  Here we summarise the different types of surfaces, groomed, well-groomed, and nice, as well as how sutured Floer homology behaves under decomposition.

 Two parallel curves or arcs $\la_1$ and $\la_2$ in a surface $S$ are said to be {\it coherently oriented} if $[\la_1]=[\la_2] \in H_1(S,\bdd S)$.

\begin{deff} \cite[Def.\,0.2]{Gabai II} \label{def groomed}
If $(M,\ga)$ is a balanced sutured manifold then a surface decomposition $(M,\ga) \leadsto^S (M',\ga')$ is called {\it groomed} if for each component $V$ of $R(\ga)$ one of the following is true:
\begin{enumerate}
\item $S \cap V$ is a union of parallel, coherently oriented, nonseparating closed curves,
\item $S \cap V$ is a union of arcs such that for each component $\de$ of $\bdd V$ we have \linebreak $\abs{\de \cap \bdd S}=\abs{\langle \de,\bdd S \rangle}$. 
\end{enumerate}
A groomed surface is called {\it well-groomed} if for each component $V$ of $R(\ga)$ it holds that $S \cap V$ is a union of parallel, coherently oriented, nonseparating closed curves or arcs.
\end{deff}

In order to define a {\it nice} surface, we need the following definition. A curve $C$ is {\it boundary-coherent} if either $[C] \neq 0$ in $H_1(R;\bZ)$, or $[C]=0$ in $H_1(R;\bZ)$ and $C$ is oriented as the boundary of the component of $R-C$ that is disjoint from $\bdd R$.

\begin{deff} \cite[Def.\,3.22]{Jupolytope} \label{def nice}
A decomposing surface $S$ in $(M,\ga)$ is called {\it nice} if $S$ is open, $v_0$ is nowhere parallel to the normal vector field of $S$, and for each component $V$ of $R(\ga)$ the set of closed components of $S \cap V$ consists of parallel, coherently oriented, and boundary-coherent simple closed curves.
\end{deff}

An important observation is that any open and groomed surface can be made into a nice surface by a small perturbation which places its boundary into a generic position.

Finally we set the groundwork for the notion of ``extremal $\Spin^c$ structure'' given in Definition \ref{deff extremal}.

\begin{deff} \cite[Def.\,1.1]{Jusurface} \label{def outer}
Let $(M,\ga)$ be a balanced sutured manifold and let $(S,\bdd S) \subset (M,\bdd M)$ be a properly embedded oriented surface.  An element $\fs \in \Spin^c(M,\ga)$ is called {\it outer with respect to $S$} if there is a unit vector field $v$ on $M$ whose homology class is $\fs$ and $v_p \neq (-\nu_S)_p$ for every $p \in S$.  Here $\nu_S$ is the unit normal vector field of $S$ with respect to some Riemannian metric on $M$.  Let $O_S$ denote the set of outer $\Spin^c$ structures.
\end{deff}

\begin{thm} \cite[Thm.\,1.3]{Jusurface} \label{thm nice}
Let $(M,\ga)$ be a balanced sutured manifold and let $(M,\ga) \leadsto^S (M',\ga')$ be a sutured manifold decomposition along a nice surface $S$.  Then 
\[
SFH(M',\ga')=\bigoplus_{\fs \in O_S} SFH(M,\ga).
\]
In particular, if $O_S$ contains a single $\Spin^c$ structure $\fs$ such that $SFH(M,\ga)\neq 0$, then
\[
SFH(M',\ga')=SFH(M,\ga,\fs).
\]
\end{thm}

\subsection{Norms on 3-manifolds} \label{subsection norms} 

Here we briefly review the various (semi)norms that have been defined on the second homology group  $H_2(M,\bdd M; \bR)$ of a 3-manifold $M$ with boundary.  This survey is meant to highlight the geometric nature of the sutured Floer polytope. 

Note that we are always given a map $H_2(M,\bdd M) \to \bZ^{\geq 0}$, which is first extended to a rational-valued map on $H_2(M, \bdd M; \bQ)$ by linearity and then to a real-valued map on $H_2(M,\bdd M; \bR)$ by continuity.  Finally, in each case some work has to be done to show that the resulting map on the real-valued homology group is indeed a seminorm.  In each case, we refer to all three maps by the same symbol, but it is obvious which one we mean. 
 
Thurston defined a seminorm on the homology of a 3-manifold $(M,\bdd M)$ with possibly empty boundary \cite{Thurston norm}.  Given a properly embedded, oriented closed surface $S \subset M$, set 
\[
\chi_-(S):= \sum_{\textrm{components $S_i$ of } S} \max\{ 0,-\chi(S_i)\}.
\]
Then the {\it Thurston seminorm} is given by the map
\begin{align*}
x &: H_2(M,\bdd M; \bR) \to \bZ^{\geq 0}, \\
x (\al) &:= \min \{ \chi_-(S) : [S]=\al \in H_2(M,\bdd M) \}.
\end{align*}
The seminorm $x$ is a norm if there exist no subspace of $H_2(M,\bdd M)$ that is spanned by surfaces of nonnegative Euler characteristic, that is, spheres, annuli and tori. The Thurston seminorm measures the ``complexity'' of a certain homology class.  Thurston showed that some top-dimensional faces of the norm unit ball are {\it fibred}.  That is, a face $F$ if fibred if there exists a fibration $M \to S^1$ with fibre $\Si$ such that the ray $r \cdot [\Si]$ for $r \in \bR^{\geq 0}$ intersects the unit ball in the interior of $F$.  Moreover, all of the rational rays through the interior of $F$ in a similar way represent fibrations of $M$.  

Scharlemann generalised the Thurston norm \cite{Scharlemann}. 
As before let $(M,\bdd M)$ be a given 3-manifold and $S$ a properly embedded surface in $M$.  Now let $\beta$ be a properly embedded 1-complex in $M$, and define 
\[
\chi_\beta(S):=\sum_{\textrm{components $S_i$ of } S} \max \{ 0, - \chi (S_i)+ \abs{S_i \cap \beta}\}.
\]
Then the {\it generalised Thurston norm} is given by the map
\begin{align*}
x_\beta & : H_2(M,\bdd M) \to \bZ^{\geq 0}, \\
x_\beta (\al) &:= \min \{ \chi_\beta(S)\mid [S]=\al \in H_2(M,\bdd M) \}.
\end{align*}

The generalised Thurston norm specialises to the case of sutured manifolds \cite{CC sutured Thurston norm, Scharlemann}.  In particular, suppose that $(M,\ga)$ is sutured manifold, and that $S$ is a properly embedded surface.  Then let $n(S)$ denote the absolute value of the intersection number of $\bdd S$ and $s(\ga)$ as elements of $H_1(\bdd M)$.  Define 
\[
\chi_-^s(S):=\sum_{\textrm{components $S_i$ of } S} \max \{ 0, - \chi (S_i)+\frac{1}{2}n(S_i)\}.
\]
Note that if we took $\be:=s(\ga)$, then $\chi_\be(S)=2 \chi_-^s(S) + \chi(S)$.  Similarly to before,  the {\it sutured Thurston norm}  is given by the map
\begin{align*}
x^s &: H_2(M,\bdd M) \to \bZ^{\geq 0}, \\
x^s (\al) &:= \min \{ \chi^s_-(S): [S]=\al \in H_2(M,\bdd M) \}.
\end{align*}
The motivation for defining $x^s$ comes from looking at the manifold $DM$ obtained by gluing two oppositely oriented copies of $M$ along the boundary, that is,  
\[
DM:=(M,\ga) \cup (-M,-\ga) / \sim,
\] 
where the equivalence relation identifies $R_+(\ga)$ with $R_+(-\ga)$ pointwise in the obvious way.  Then $DM$ is referred to as the {\it double} of $M$.  Similarly, if $S$ is a properly embedded surface in $M$, then $DS$ is the double of $S$ in $DM$.    Now, Theorem 2.3 \cite{CC sutured Thurston norm} says that there is a natural ``doubling map'' $D_* \colon H_2(M,\bdd M;\bR) \to H_2(DM,\bdd DM;\bR)$, so that for any $\al \in H_2(M,\bdd M;\bR)$, we have 
\[
x^s(\al)=\frac{1}{2}x(D_*(\al)).
\]

So far all of the described seminorms have been symmetric.  As the sutured Floer polytope is asymmetric in general, the unit balls of these seminorms are certainly not dual to the sutured Floer polytope. Also, in order to talk about the polytope as being dual to the unit ball of a seminorm, we must pick a trivialisation because otherwise the polytope is defined only up to translation in $H^2(M,\bdd M)$.  

Fix a balanced sutured manifold $(M,\ga)$ and a trivialisation $t \in T(M,\ga)$.  Using the theory developed by Juh\'asz, we define an integer-valued function on $H_2(M,\bdd M)$, dependent on $t$, that plays the role of the Thurston-type norms in the case of the sutured Floer polytope.  First of all, associate an integer $c(S,t)$ to an oriented decomposing surface $S$ in $(M,\ga)$ \cite[Def.\,3.16]{Jupolytope}.  This purely geometric invariant is given by
\begin{equation} \label{equation cst}
c(S,t):= \chi(S) + I(S) - r(S,t),
\end{equation}
where $\chi(S)$ is the Euler characteristic, $I(S)$ generalises the term $-\frac{1}{2}n(S_i)$ in the definition of the sutured seminorm, and $r(S,t)$ is an additional component, which accounts for the dependance of the polytope on the trivialisation $t$.  

Any generic oriented decomposing surface $S$ is such that the positive unit normal field $\nu_S$ of $S$ is nowhere parallel to $v_0$ along $\bdd S$.  Denote the components of $\bdd S$ by $T_1, \ldots, T_k$; each of the components has an orientation coming from the orientation of $S$.  Let $w_0$ denote the nowhere zero vector field obtained by projecting $v_0$ into $TS$.  Further, let $f$ be the positive unit tangent vector field of $\bdd S$.  For $1 \leq i \leq k$, define the {\it index} $I(T_i)$ to be the (signed) number of times $w_0$ rotates with respect to $f$ as we go around $T_i$.  Then set
\[
I(T_i):=\sum_{i=1}^k I(T_i).
\]

Next, let $p(\nu_S)$ be the projection of $\nu_S$ into $v_0^\perp$.  Observe that $p(\nu_S)|\bdd S$ is nowhere zero.  For $1 \leq i \leq k$ define $r(T_i,t)$ to be the number of times $p(\nu_S)|\bdd T_i$ rotates with respect to $r$ as we go around $T_i$.  Then set
\[
r(S,t):=\sum_{i=1}^k r(T_i,t).
\]
  
Now, assuming that $H_2(M)=0$ and for a fixed $t \in T(M,\ga)$ we define the function
\begin{align*}
y_t & : H_2(M,\bdd M) \to \bZ, \\
y_t(\al)& := \min \{ -c(S,t)): S \textrm{ nice decomposing surface}, [S]=\al \}.
\end{align*}
\begin{rmk} As we noted before, any open groomed surface can be slightly perturbed into a nice surface.  Any homology class $\al \neq 0$ has a groomed surface representative \cite[Lem.\,0.7]{Gabai II}, however it is not clear that it necessarily has an open groomed representative.  Thus the condition $H_2(M)=0$.   We could have relaxed the definition and required each $S$ to satisfy all the conditions of being nice except openness, but it is not clear that this would have been helpful.
\end{rmk}

It can be shown that if $T$ is a component of $\bdd S$ such that $T \not \subset \ga$ then $I(T)=-\frac{\abs{T \cap s(\ga) }}{2}$ \cite[Lem.\,3.17]{Jupolytope}.  In other words, in this case $-I(T)=\frac{1}{2} n(T)$ which is the second term in the definition of $x^s$.

We would like to say that the function $y_t$ has some useful properties, such as that it satisfies the triangle inequality and positive homogeneity with respect to the integers. Indeed, as we see in Proposition \ref{prop properties}, these properties follow from the definitions and from Lemma \ref{lem calt}.  

\begin{lemma} \cite[Cor.\,4.11]{Jupolytope} \label{lem calt}
Let $(M,\ga)$ be a taut, strongly balanced sutured manifold such that $H_2(M)=0$.  Then 
\begin{equation} \label{equation calt2}
c(\al,t)=\max \{ c(S,t) : S \textrm{ a nice decomposing surface}, [S]=\al\}.
\end{equation} 
\end{lemma}

\begin{prop} \label{prop properties}
Let $(M,\ga)$ be a taut, strongly balanced sutured manifold such that $H_2(M)=0$.  Fix $t \in  T(M,\ga)$. Then for any $\al, \be \in H_2(M,\bdd M)$ and any $m \in \bZ$, the following hold
\begin{align*}
y_t(\abs{m} \cdot \al) & = \abs{m} \cdot y_t(\al ), \\
y_t(\al + \be)  & \leq y_t(\al) + y_t(\be).
\end{align*}
\end{prop}
\begin{proof}
From Lemma \ref{lem calt} it follows that $y_t(\al)=-c(\al,t)$ for any  $\al \in H_2(M,\bdd M)$.  By definition $c(\al,t)=\min \{ \langle c, \al \rangle : c \in P(M,\ga,t) \}$, thus 
\[
y_t(\al)=\max \{\langle -c, \al \rangle : c \in P_t\},
\]
where we have denoted $P(M,\ga,t)$ by $P_t$.

The first statement of the proposition is obvious. Proving the triangle inequality is also easy, and is identical to the proof given in  \cite[Prop.\,8.2]{Jupolytope}:
\begin{align*}
y_t(\al + \be) & =\max \{ \langle -c, \al + \be \rangle : c \in P_t \} \\
		& =\max \{ \langle -c, \al \rangle + \langle -c, \be \rangle : c \in P_t \}  \\
		& \leq \max \{ \langle -c, \al  \rangle : c \in P_t \} + \max \{ \langle -c, \be \rangle : c \in P_t \} \\
		& = y_t(\al) + y_t(\be).
\end{align*} 
\end{proof}

As before we can extend $y_t$ to a rational-valued map on $H_2(M,\bdd M ;\bQ)$ by linearty and then to a real-valued map on $H_2(M,\bdd M; \bR)$ by continuity.  Thus, for any balanced sutured manifold with $H_2(M)=0$ we can define a {\it geometric sutured function}
\[
y_t \colon H_2(M,\bdd M) \to \bR,
\]
such that $y_t(r \cdot \alpha)=r \cdot y_t(\al)$ and $y_t(\al + \be) \leq y_t(\al) + y_t(\be)$ for $r \in \bR$ and $\al, \be \in H_2(M,\bdd M ;\bR)$.

The following corollary says that $y_t$ is actually a (semi)norm for a lot of the often-studied sutured manifolds.

\begin{cor} \label{yt dual to polytope}
Let $(M,\ga)$ be a taut, strongly balanced sutured manifold such that \linebreak $H_2(M)=0$.  If there exists a $t \in T(M,\ga)$ such that 
\[
y_t \colon H_2(M,\bdd M) \to \bR^{\geq 0}, \\
\]
then $y_t$ is an asymmetric seminorm. In particular, this is the case when $H^2(M)=0$. Moreover, the unit ball of the seminorm $y_t$ is the dual to the polytope $P(M,\ga,t)$.  Finally,  $y_t$ is a norm if and only if  $\dim P(M,\ga,t) = b_1(M)$.
\end{cor}

The fact that such a $t$ exists follows from Lemma 3.12 of \cite{Jupolytope}.  Basically, when there is no torsion in $H_1(M)$, then we can choose a trivialisation such that $P(M,\ga,t)$ contains $0 \in H^2(M,\bdd M)$. The very last statement in the corollary uses the same argument as in the proof of \cite[Prop.\,8.2]{Jupolytope}.

\begin{rmk} \label{rmk Ju norm}
Juh\'asz defines an asymmetric seminorm $y$ whose dual seminorm unit ball is $-P(M,\ga)$, where  $-P(M,\ga)$ is the centrally symmetric image of $P(M,\ga)$  \cite[Def.\,8.1]{Jupolytope}.  Here $P(M,\ga)$ is the polytope with the centre of mass at $0 \in H^2(M,\bdd M;\bR)$.  Specifically he defines
\begin{align*}
y &: H_2(M,\bdd M ; \bR) \to \bR^{\geq 0} \\
y(\al) &:=\max \{\langle -c ,\al \rangle : c \in P(M,\ga) \}.
\end{align*}
When $t$ is such that the centre of mass of $P(M,\ga,t)$ lies at $0\in H^2(M,\bdd M)$, then $y_t=y$.
\end{rmk}

\section{Foliations on sutured manifolds} \label{section foliations}

This section sets up the foliation theory necessary to understand our duality result.  It begins with some basic theory of foliations on sutured manifolds; further background reading can be found in \cite{Foli1,Foli2}.  We then proceed to describe the way Gabai constructs taut, finite depth foliations on any given sutured manifold \cite{Gabai}.  This is followed by a description of junctures and spiral staircase neighbourhoods, which are used in proving one direction of Lemma C.  The last subsection introduces the foliation cones of Cantwell and Conlon \cite{CC99}.

\subsection{Foliations on sutured manifolds}\label{subsection foliations}

\begin{deff}
A {\it foliation $\mF$ of a sutured manifold $(M,\ga)$} is a foliation of $M$ such that the leaves of $\mF$ are transverse to $\ga$ and tangential to $R(\ga)$ with normal direction pointing inward along $R_-(\ga)$ and outward along $R_+(\ga)$.  
\end{deff} 

The collection of simple closed curves $\bdd \ga \subset \bdd M$ can be considered as a set of convex corners on $M$. Elsewhere in the literature, when $M$ is any given manifold, the notation $\bdd _\pitchfork M$ is used to denote the subset of $\bdd M$ where the leaves of $\mF$ are transverse to the boundary of $M$, and $\bdd_\tau M$ to denote the complement $\bdd M - \bdd_\pitchfork M$ where the leaves of $\mF$ are tangential to the boundary of $M$.

Let $L$ be a leaf of a foliation $\mF$ and let $\{C_\al\}_{\al \in \mA}$ denote the family of all compact subsets of $L$.  Let $W_\al:=L -C_\al$ and denote by $\overline W_\al$ the closure of $W_\al$ in $M$.  Recall that the {\it asymptote} of $L$ is the limit set $\lim L:=\bigcap_{\al \in \mA} \overline W_\al$.  The asymptote for a noncompact leaf $L$ is a compact, nonempty, $\mF$-saturated set, that is, $\lim L$ is a compact set which is a nonempty union of leaves of $\mF$.  Further,  $L$ is said to be {\it proper} if it is not asymptotic to itself, and {\it totally proper} if every leaf in the closure $\overline L$ is a proper leaf.  

A leaf $L$ of a foliation $\mF$ is said to be at {\it depth 0} if it is compact.  For $k>1$, a leaf $L$ is said to be at {\it depth k} if $\overline L \setminus L$ is a collection of leaves at depths less than $k$, with at least one leaf at depth $k-1$ \cite[Def.\,8.3.14]{Foli1}.  Note that this definition of depth  assumes that the leaf is totally proper.  A {\it depth $k$ foliation} is a foliation with all leaves at depth $l \leq k$ and at least one leaf at depth $l=k$.

We are primarily interested in taut depth one foliations $\mF$ of a sutured manifold $(M,\ga)$.  As is described below, these foliations are built from the product foliation $\mP$ of a product sutured manifold $(M',\ga'):=(\Si \times I, \bdd \Si \times I)$. It is sometimes necessary to specify a particular leaf of $\mP$, and so we assume that $\mP$ is precisely the foliation with leaves $L_t:=\Si \times t$, for $t\in I =[0,1]$.

\begin{deff}  \label{deff taut}
A transversely oriented codimension one foliation $\mF$ on a sutured manifold $(M,\ga)$ is {\it taut} if there exists a curve or properly embedded arc in $M$ that is transverse to the leaves of $\mF$ and that intersects every leaf of $\mF$ at least once.

\end{deff}

\begin{convention}
Let $(M,\ga)$ be a sutured manifold.  From now on, when we say a depth one or depth zero foliation of $(M,\ga)$, it is implicit that these foliations are smooth, transversely oriented, and in the case of the depth one their sole compact leaves are the connected components of $R(\ga)$.  
\end{convention}  

\begin{rmk} \label{rmk depth zero and foliations} Suppose that $M$ is a manifold with boundary that fibres over $S^1$.  Then $\bdd M$ is a (possibly empty) collection of tori.  If $M$ is made into a sutured manifold by specifying the sutures $\ga$, then one of two situations occurs: 
\begin{enumerate}
\item $\ga=T(\ga)=\bdd M$ and the fibration is a depth zero foliation;
\item $A(\ga) \neq \emptyset$, and the fibres of the fibration are {\it not} transverse to $\ga$ or {\it not} tangential to $R(\ga)$.
\end{enumerate}
\end{rmk}

Set $M_0:=M \setminus R(\ga)$.   The following lemma summarises useful information about depth one foliations.

\begin{lemma} \cite[Lem.\,11.4.4]{Foli2}\label{lemma fibring}
Let $\mF$ be a transversely oriented, $C^\infty$ foliation of the connected sutured manifold $(M,\ga)$, transverse to $\ga$ and having the components of $R(\ga)$ as sole compact leaves.  Let $\mL$ be a smooth one-dimensional foliation transverse to $\mF$ and tangent to $\ga$, so that $\mL|A(\ga)$ is a foliation by compact, properly embedded arcs.  Then the following statements are equivalent.
\begin{enumerate}
\item $\mF$ is a taut depth one foliation. 
\item There is a smoothly embedded circle $\Si \subset M_0$ that is transverse to $\mF | M_0$, meeting each leaf of that foliation exactly once.
\item $\mL$ can be chosen to have a closed leaf in $M_0$ that meets each leaf of $\mF| M_0$ exactly once.
\item $\mF|M_0$ fibres $M_0$ over $S^1$.
In this case, there is a $C^0$ flow $\Phi_t$ on $M$ having the leaves of $\mL$ as flow lines, stationary at the points of $R(\ga)$, smooth on $M_0$ and carrying the leaves of $\mF$ diffeomorphically onto one another. 
\end{enumerate}
Let  $\mF$ be a depth one foliation on $(M,\ga)$.  Then $\mF$ determines a fibration $M_0 \to S^1$, with fibres the noncompact leaves of $\mF$.
\end{lemma}

Following \cite[Sec.\,2]{CCisotopy}, we associate to each depth zero or depth one foliation a cohomology class $H^1(M)$ using Lemma \ref{lemma fibring}.  In particular, the fibration of $M_0$ over $S^1$ gives rise to a map 
\begin{equation} \label{lambda}
\la(\mF) \colon \pi_1(M) \to \bZ = \pi_1(S^1),
\end{equation}
which passes to a cohomology class $\la(\mF) \in H^1(M)$.  

\begin{deff} \label{deff foliation equivalence}
Two depth zero or depth one foliations $\mF$ and $\mF'$ are said to be {\it equivalent}, denoted by $\mF \sim \mF'$, if $\mF$ is isotopic to $\mF'$ via a continuous isotopy that is smooth in $M_0$.  
\end{deff}

The following theorem says  that an equivalence class of foliations $\mF$ is uniquely determined by $\la(\mF)$. 

\begin{thm} \cite[Thm.\,1.1]{CCisotopy} \label{thm equivalence}  Let $\mF$ and $\mF'$ be depth one  type foliations of $(M,\ga)$, such that $\la(\mF)=\la(\mF')$.  Then $\mF$ is equivalent to $\mF'$.
\end{thm}

Note that Theorem \ref{thm equivalence} is also true when the foliations are of depth zero, that is, when they are fibrations of $M$ with fibres transverse to $\ga$.

\subsection{Gabai's construction of depth one foliations} \label{subsection Gabai's construction}

In the well-known paper \cite{Gabai}, Gabai gives a way of constructing finite depth foliations on sutured manifolds from a sutured manifold hierarchy.  A  {\it sutured manifold hierarchy} is a sequence of decompositions along surfaces 
\[
(M_0,\ga_0) \leadsto^{S_1} (M_1,\ga_1) \leadsto^{S_2} \cdots \leadsto^{S_n} (M_n,\ga_n),
\]
where $(M_n,\ga_n)$ is a product sutured manifold.  

\begin{thm} \cite[Thm.\,5.1]{Gabai} \label{thm Gabai construction}
 Suppose $M$ is connected and $(M,\ga)$ has a sutured manifold hierarchy
 \[
(M,\ga)=(M_0,\ga_0) \leadsto^{S_1} (M_1,\ga_1) \leadsto^{S_2} \cdots \leadsto^{S_n} (M_n,\ga_n),
\]
so that no component of $R(\ga_i)$ is a compressing torus.  Then there exist  transversely oriented foliations $\mF$ and $\mG$ of $M$ such that the following hold.
\begin{enumerate}
\item $\mF$ and $\mG$ are tangent to $R(\ga)$.
\item $\mF$ and $\mG$ are transverse to $\ga$.
\item If $H_2(M,\ga) \neq 0$, then every leaf of $\mF$ and $\mG$ nontrivially intersects a transverse closed curve or a transverse arc with endpoints in $R(\ga)$.  However, if $\emptyset \neq \bdd M =R_+(\ga)$ or $R_-(\ga)$, then this holds only for interior leaves.
\item There are no 2-dimensional Reeb components of $\mF|\ga$ and $\mG|\ga$.
\item $\mG$ is $C^\infty$ except possibly along toroidal components of $R(\ga)$ or along toroidal components of $S_1$ if $\bdd M = \emptyset$.
\item $\mF$ is of finite depth.
\end{enumerate}
\end{thm}

The ideas from this proof have been exploited a lot; a good summary of the construction is given in \cite[Sec.\,5.6]{Calegari}.  However, understanding the  construction in the case when $n=1$ is essential for the proof of Lemma C, so we explain the ideas in this simplified case.

Let $(M,\ga)$ be a sutured manifold with a properly embedded surface $S$ in $M$, such that $(M,\ga) \leadsto^S (M',\ga')$ gives a product manifold $(M',\ga')$.  Gabai's construction proceeds as follows: first, he shows that there is a series of decompositions along well-groomed surfaces $S_i \subset M_{i-1}$ giving the same product manifold
\[
(M,\ga) \leadsto^{S_1} (M_1,\ga_1) \leadsto^{S_2} \cdots \leadsto^{S_n} (M',\ga');
\]
second, he demonstrates how to produce a $C^0$ foliation $\mF_{k-1}$ on $(M_{k-1}, \ga_{k-1})$ from a finite depth $C^0$ foliation $\mF_{k}$ on $(M_k,\ga_k)$, such that $\depth(\mF_{k-1}) =\depth (\mF_k)+1$.  This means that if  $S$ is well-groomed, starting from the product foliation on $(M',\ga')$ we have a recipe how to construct a taut depth one $C^0$ foliation $\mF=\mF_0$ on $(M,\ga)$.  

\begin{rmk} When we have $(M,\ga) \leadsto^{S_1} (M_1,\ga_1)$, for $(M_1,\ga_1)$ a product manifold, then the construction of $\mF$ coincides with the construction of $\mG$.  This means that $\mF$ is $C^\infty$ except possibly along toroidal components.  However, as we are starting from the product foliation on $(M_1,\ga_1)$, actually $\mF$ is smooth along toroidal components as well.
\end{rmk}

So far we have that if $S$ is well-groomed   the ``only if'' direction of Lemma C follows straight from Theorem \ref{thm Gabai construction}.  Thus, the content of the proof of the ``only if'' direction of Lemma C below is the removal of the condition for $S$ to be well-groomed.  When we remove this condition in Lemma \ref{lemma no conditions}, we do not use Gabai's method of showing that any decomposition can be broken down into the well-groomed ones. Thus we omit the details of this first step in his proof; see Theorem \cite[Thm.\,5.4]{Gabai}.  However, we explain the remaining steps of his construction, as it is crucial to the rest of the paper.

Let us focus on constructing $\mF$ from a product foliation $\mP$ on $(M',\ga')$ when $S$ is well-groomed and nonempty.  Let $V$ be a component of $R(\ga)$.  Recall that $S$ being well-groomed is a condition on $\bdd S \cap V$; in particular, that $\bdd S \cap V$ is a (possibly empty) collection of  parallel, coherently oriented, nonseparating closed curves or arcs.    The construction describes a process of {\it spinning} $S$ near $V$.  It suffices to show how this process works near $V$, as it is analogous for all other components.

Consider each of the following cases: first, $\bdd S \cap V$ is empty; second, $\bdd S \cap V$ is a collection of closed curves; third, $\bdd S \cap V$ is a collection of arcs. For definiteness assume that $V \subset R_+(\ga)$.  
\\ 

\noindent {\bf Case 1.}  Suppose that $\bdd S \cap V$ is empty for some component $V $. Then $V$ is a component of $R_+(\ga')$, but since $(M',\ga')$ is a product manifold $(\Si \times I, \bdd \Si \times I)$, it follows that $V = R_+(\ga')$. As $S_+$ must be contained in $R_+(\ga')$, it follows that $S=\emptyset$, which is a contradiction.  An analogous argument works for $V \subset R_-(\ga)$.  Thus, either $\bdd S \cap V \neq \emptyset$ for all components $V\subset R(\ga)$ or $R(\ga)=\emptyset$.  In the latter case $\ga=T(\ga)$ and $R_\pm(\ga') \cong S_\pm$.  Then $M$ fibres over $S^1$ with fibre homeomorphic to $S$ and $\mF$ is the fibration.
\\

\noindent {\bf Case 2.} Suppose that $\bdd S \cap V$ is a collection of parallel, coherently oriented, nonseparating closed arcs.  Denote by $\mJ$ the collection of arcs of $\bdd S \cap V$ and label the arcs $J_1, \ldots, J_l$ starting with one outermost arc, and proceeding in the obvious way.  Subsection \ref{subsection junctures} gives the motivation for using the letter $J$, which stands for {\it juncture}. We first explain how to construct $\mF$ near $V$ when $l=1$.  It is then fairly clear how to extend the construction to the general case.

\begin{figure}[h]
\centering
\includegraphics [scale=0.50]{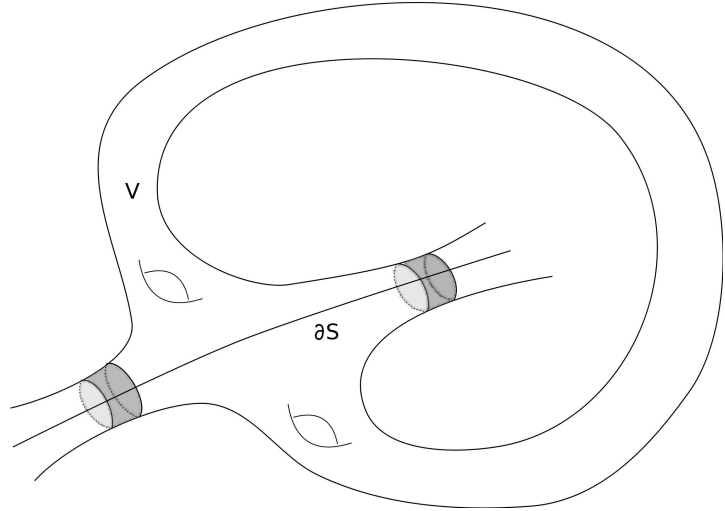}
\caption{Example where $\bdd S \cap V$ is a single arc.  The shaded annuli are two annular components of $\ga$.}
\label{fig ArcOnSurface}
\end{figure}

Suppose that $\bdd S \cap V$ is a single arc $J$; for an example see Figure \ref{fig ArcOnSurface}.   When decomposing along $S$, we remove a neighbourhood $N(S) \cong S\times I$, where $S \times \bdd I=S_+ \cup S_-$; see Definition \ref{def decomp} for the labelling convention.  Define $J^+:= S_+ \cap V$ and $J^-:=S_-\cap V$.  Consider a neighbourhood $N:=J \times (-2,0)$ of $J$ in $V$, parametrised so that $J^-=J \times 0$ and $J^+=J \times -1$ (Figure \ref{fig JNeighbourhood}). The reason for parametrising an interval as $(-2,0)$ becomes clear shortly.

\begin{figure}[h]
\centering
\includegraphics [scale=0.6]{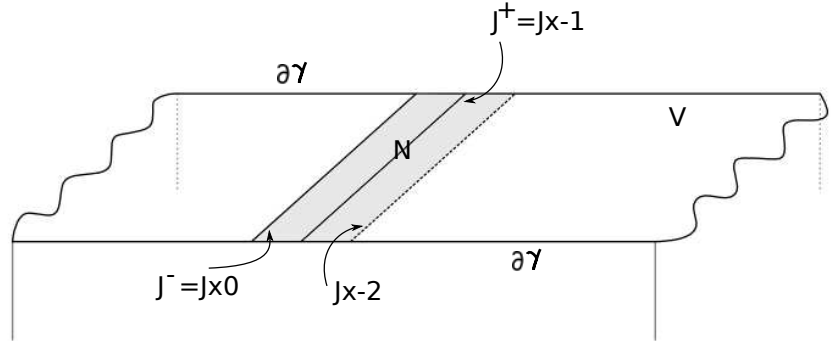}
\caption{A parametrised neighbourhood of $\bdd S \cap V=J$.}
\label{fig JNeighbourhood}
\end{figure}

Start with the product foliation $\mP$ on $(M',\ga')$.  Glue $(M',\ga')$ along $S_+$ and $S_-$ to recover $(M,\ga)$, in such a way that $J^+$ is a concave corner, whereas $J^-$ is a convex corner in $V$ (Figure \ref{fig ConvexConcave}).  
\begin{figure}[h]
\centering
\includegraphics [scale=0.6]{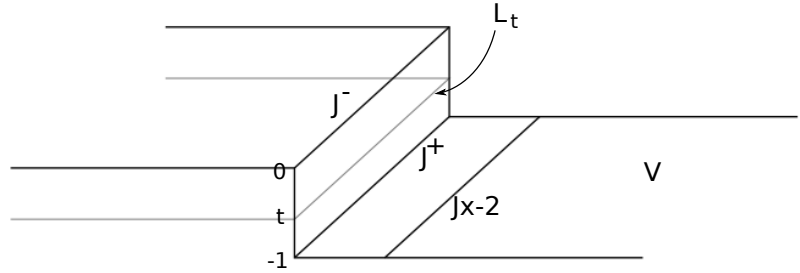}
\caption{A leaf $L_t$ of the product foliation $\mP$ is transverse to $V$ after the regluing of $S_+$ and $S_-$.  }
\label{fig ConvexConcave}
\end{figure}  
  
Let $K$ be the product manifold $(V \setminus N) \times [0,\infty]$ with the product foliation; that is, the leaves are of the form $(V \setminus N) \times s$, for $s \in [0,\infty]$.  The manifold obtained by gluing $K$ to $M$ in the obvious way
 \[
M \cup_{(V \setminus N) \sim (V\setminus N) \times 0}  K
\] 
is homeomorphic to $M$, but it has a ``ditch'' with two ``walls'' $W_0$ and $W_1$ (Figure \ref{fig ConstructingTheFoliation}).  Note that $W_0= J  \times -1 \times [-1,\infty]$ and $W_1=J \times -2 \times [0,\infty]$.  It is now evident why we chose the parametrisation $(-2,0)$: so that the last coordinate of $W_0$ is parametrised as $[-1, \infty]$.

\begin{figure}[h]
\centering
\includegraphics [scale=0.60]{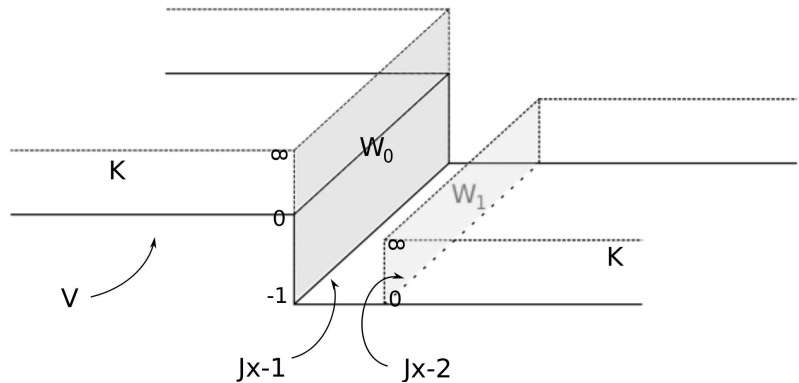}
\caption{The manifold $M \cup K \cong M$.}
\label{fig  ConstructingTheFoliation}
\end{figure}  

There is now a foliation $\mF'$ of $M \cup K$ obtained from $\mP$ union the product foliation on $K$, however it is not a foliation of the sutured manifold $(M,\ga)$.  The leaves of $\mF'$ are transverse to the walls and tangential everywhere else on $V$.  It remains to glue a product manifold $D$ parametrised as $J \times [-1,-2] \times [0,\infty]$ with a product foliation whose leaves are parallel to $J \times [-2,-1]$.  The gluing is done by making pointwise identifications:
\begin{enumerate}
\item $W_0 \sim J \times -1 \times [0,\infty]$ via the equivalence relation
\begin{equation} \label{shift one}
(x ,-1, t) \sim (x, -1, t+1),
\end{equation}
\item $J \times -2 \times [0,\infty] \sim W_1$ via the (identity) equivalence relation $(x, -2,t) \sim (x, -2, t).$
\end{enumerate}
So the ditch has been filled in with a product neighbourhood in such a way that the new foliation $\mF$ is of depth one and is tangent to $V$.

Note that the leaf $L_t \subset \mP$ from Figure \ref{fig ConvexConcave} has now become a leaf that ``spirals'' onto $V$.  See Figure \ref{fig TheLeaf} for a two-dimensional representation of what $L_t$ looks like near $J$ as it spirals onto $V$.
\begin{figure}[h]
\centering
\includegraphics [scale=0.9]{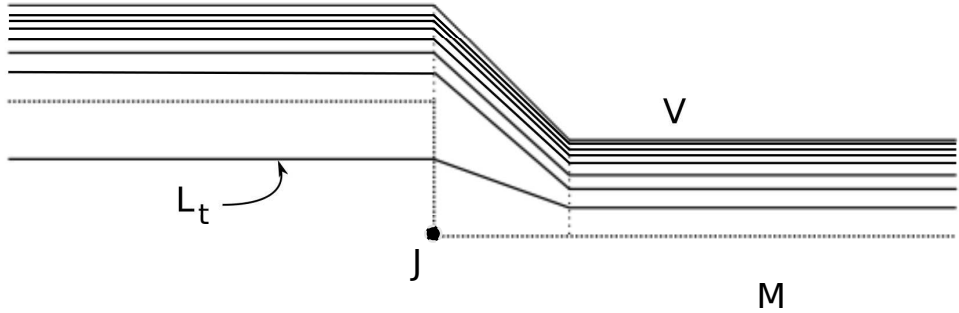}
\caption{A depth one leaf $L_t$ of the $\mF$ spiralling onto $V$.}
\label{fig  TheLeaf}
\end{figure}  

The identification \eqref{shift one} ensures that two interior leaves $L_t$ and $L_s$ of $\mP$ are {\it spun} onto $V$ in such a way that their distance along a transverse arc remains the same, and hence there can be no holonomy along the noncompact leaves.   This means that we have constructed the (depth one) foliation $\mF$ of $(M,\ga)$. The manifold $K \cup J \times [-1,-2] \times [0,\infty]$ is  a so called {\it spiral staircase neighbourhood} of $V$ for the foliation $\mF$ as is explained in Subsection \ref{subsection junctures}.

In the general case, since the construction takes place in a collar neighbourhood of a component of $R(\ga)$, it can be repeated for any number of components $V$ for which $\bdd S \cap V$ is nonempty and the intersection consist of arcs.  Further, if $\mJ=\bdd S \cap V $ contains $l >1$ arcs, then a simple modification of the above construction yields a depth one foliation $\mF$ of $(M,\ga)$.  More specifically, it suffices to replace $J$ by $\cup_{i=1}^l J_i$ and to repeat the same gluings of the walls in each of  the $l$ ditches $D_i$, labeled after their associated junctures $J_i$.  The key observation is that since the arcs are well-groomed and the identification in \eqref{shift one} is specified to induce no holonomy, the depth of each noncompact leaf is one.  Then the spiral staircase neighbourhood becomes a ``staircase'' indeed, with $l$ steps.
\\ 

\noindent {\bf Case 3.}  Suppose that $\bdd S \cap V$ is a collection of parallel, coherently oriented, nonseparating closed curves.  The same construction as described above works, only the $J_i$-s are now going to be parallel closed curves, as opposed to arcs.

\begin{rmk}
The described construction slightly differs from the one given by Gabai originally in \cite[Thm,\,5.1]{Gabai}.  For closed curves instead of gluing a copy of 
\[
(V \setminus \cup_i J_i \times (-1,1)) \times [0,\infty],
\]
 he finishes off the foliation using a ``1/2 infinity cover'' of $V$. Here it is more convenient to use the ``creating ditches'' method in both cases, as it is easier to see why the well-groomed condition can be removed.
\end{rmk}

From now on we refer to the construction described in this subsection as {\it Gabai's construction}.

\subsection{Junctures and spiral staircases} \label{subsection junctures}

 In the previous subsection we discussed Gabai's method for constructing a taut depth one foliation $\mF$ from a well-groomed surface decomposition $(M,\ga) \leadsto (M',\ga')$ that results in a product $(M',\ga')$. This construction is crucial in proving the ``only if'' direction of Lemma C.  The connected components of the 1-manifold $\bdd S \cap \bdd M$ are called the {\it junctures} of the foliation $\mF$.  We now explain the role that junctures play in the spiralling of leaves of a depth one foliation along the components of $R(\ga)$, as understanding this is crucial in proving the ``if'' direction of Lemma C.  The following terminology and theory can be found in \cite{Foli1} and \cite[Sec.\,2]{Handel-Miller}.

  Let $\{K_\al\}_{\al\in A}$ be the family of all compact subsets of $L$, and let $\{U_\al\}_{\al \in A}$ be the family of sets $U_\al:=L \setminus K_\al$.  Then consider descending chains $\{U_{\al_i}\}_{i=1}^\infty$ of the  form
\[
 U_{\al_1} \supsetneq U_{\al_2} \supsetneq \cdots \supsetneq U_{\al_n} \supsetneq \cdots,
\]
which satisfy the following condition 
\[
\bigcap_{i=1}^\infty U_{\al_k}=\emptyset.
\]
Two such chains, $\mU:=\{U_{\al_i}\}_{i=1}^\infty$ and $\mV:=\{V_{\be_i}\}_{i=1}^\infty$ are related, and their relation denoted by $\mU \sim \mV$, if for each $i \geq 1$ there exists an $n >i$ such that 
\[
U_{\al_i} \supset V_{\be_n} \textrm{ and } V_{\be_i} \supset U_{\al_n}.
\]
Clearly this is an equivalence relation.  An equivalence class of such descending chains is called an {\it end} of $L$.  Any descending chain $\mU$ that belongs to the equivalence class of an end $e$ is said to be a {\it fundamental open neighbourhood system} of $e$.   If an open set $U \subset L$ contains a fundamental open neighbourhood system of $e$, then $U$ is said to be a {\it neighbourhood} of $e$. Denote by $\mE(L)$ the set of ends of $L$.

Let $f \colon L \to L$ be a homeomorphism.  Then an end $e$ of $L$ is said to be {\it cyclic of period $p$} if $f^p(e)=e$ and $p$ is the least such integer.  An end $e$ is called {\it attracting} ({\it repelling}) if there is a neighbourhood $U$ of $e$ and an integer $n>0$ ($n<0$) such that $f^n(\overline U) \subset U$ and $\bigcap_{i=1}^\infty f^{n}(U) =\emptyset$, where $\overline U$ denotes the closure of $U$ in $L$.  Attracting and repelling ends are called {\it periodic ends}.  

It is clear that periodic ends are cyclic, but the converse is not be true in general; see \cite[p.\,4]{Handel-Miller}.  If a homeomorphism $f \colon L \to L$ is such that all cyclic ends are periodic, then $f$ is called an {\it endperiodic homeomorphism} of $L$.

\begin{lemma} \label{lemma cyclic}
Let $\mF$ be a taut depth one foliation of the sutured manifold $(M,\ga)$.  Then every depth one leaf $L$ has finitely many ends.  Hence, for any homeomorphism $f \colon L \to L$ every end $e \in \mE(L)$ is cyclic. 
\end{lemma}

\begin{proof}   Let $\pi \colon M_0 \to S^1$ be the fibration induced by $\mF$; see Lemma \ref{lemma fibring}.  For every component $\de$ in $\ga$, the foliation $\mF$ induces on $\de$ either a fibration or a depth one foliation $\mF_{\de}$ with two compact leaves and noncompact leaves homeomorphic to the open interval $(0,1)$.   The number of ends of a noncompact leaf $L$ is equal to the number of components of $L \cap \de$ summed over all $\de \subset \ga$ when $\mF_\de$ is of depth one. Note that if $\de \subset T(\ga)$ then $\mF_\de$ is a fibration.

Let $\de \subset A(\ga)$ and so $\mF_\de$ is of depth one.  Let $a$ be an arc transverse to $\bdd \de$ with one endpoint in each component of $\bdd \de$.  Let $l$ be a depth one leaf of $\mF_\de$.  Then $l \cap a$ is a collection of points $\mA$ with limit points at $\bdd a$.  Choose two points $x,y \in \mA$, so that the line segment from $x$ to $y$ in $a$, denoted by $xy$, contains no other points of $\mA$.  Then any other leaf $l'$ of $\mF_\de$ intersects $xy$ in precisely one point.  So if a noncompact leaf $L$ of $\mF$ intersects $\de$ in infinitely many leaves of $\mF_\de$, then $\abs{\bdd L \cap xy}$ is also a infinite collection of points.  As $xy$ is a compact interval, the points $\{ \bdd L \cap xy\}$ have a limit point in the interval $xy$, thus contradicting the fact that $\mF_\de$ is of depth one.
\end{proof}

Let $\mL$ be a 1-dimensional foliation transverse to a taut depth one foliation $\mF$ of a sutured manifold $(M,\ga)$. All smooth, codimension-1 foliations $\mF$ admit at least one such transverse foliation $\mL$.  If $L$ is a depth one leaf, define $f \colon L \to L$ to be the first return map given by $\mL_f:=\mL$.  Then $f$ is {\it endperiodic}, that is, all of the (cyclic) ends of $L$ are periodic \cite{Fenley}.  In combination with Lemma \ref{lemma cyclic} this means that all ends of a taut depth one foliation are periodic.

Next, we review some terminology related to ``spiralling'' in of leaves and ends.  Fix a depth one leaf $L$.  As before let $K_\al$ be the family of all compact subsets of $L$.   Let $U_\al:=L \setminus K_\al$, and denote by $\overline U_\al$ the closure of $U_\al$ in $M$.  Then the set $\lim L:=\bigcap_{\al \in A} \overline U_\al$ is a compact, non-empty, $\mF$-saturated set \cite[Lem.\,4.3.2]{Foli1}.  Thus we say that $L$ is {\it asymptotic} to a leaf $F$ if $F \subset \lim L$.  

Similarly, if $\mU:=\{U_{\al_i}\}_{i=1}^\infty$ is a fundamental open neighbourhood system of an end $e$ of $L$, and $\overline U_{\al_i}$ the closure of $U_{\al_i}$ in $M$, then $\lim_e L:= \bigcap_{i=1}^\infty \overline U_{\al_i}$.  Again, \cite[Lem.\,4.3.5]{Foli1} says that for each leaf $L$ of $\mF$ and each end $e \in \mE(L)$, the set $\lim_e L$ is a compact, nonempty $\mF$-saturated set, not depending on the choice of fundamental neighbourhood system of $e$.  Thus we also say that the end $e$ is {\it asymptotic} to a leaf $F \subset \lim_e L$.

Finally, we arrive at the definition of junctures.  Let $e$ be an attracting, cyclic end of $L$ of period $n_e$, and let $\mU=\{U_i\}_{i=1}^\infty$ be a fundamental open neighbourhood system.  Define $\overline U_i$ to be the closure of $U_i$ in $L$.  From the definition of a neighbourhood of $e$, we know that the closure of $\bdd U_i \setminus \bdd L$ is a compact 1-manifold that separates $L$.  It follows that the sets $B_i:= \overline U_i \setminus U_{i+1}$ are closed, not necessarily compact, subsurfaces of $L$, such that $\bdd B_i \setminus \bdd L$ is also a compact 1-manifold.  Moreover, since the end is cyclic, it follows that $f^{n_e}(B_i) =B_{i+1}$ and that $\overline U_i = B_i \cup B_{i+1} \cup \cdots, \forall i \geq 0$.  The subsurfaces $B_i$ are called {\it fundamental domains} for the attracting end.  

\begin{deff}
The compact 1-manifolds $J_i:=\overline U_i \setminus U_i$, for $0 \leq i < \infty$ are called positive (negative) junctures for an attracting (repelling) end $e$. 
\end{deff}

Note that the connected components of junctures can be both closed curves and properly embedded arcs.  

Actually, one can also define a set of curves on $R(\ga)$ called junctures.  Let $V$ be a component of $R(\ga)$.  As before, suppose that $\mF$ is a taut depth one foliation of $(M,\ga)$, so certainly $V$ is a compact leaf of $\mF$.  Let $e$ be an end of a depth one leaf $L$ which is asymptotic to $V$.  Then if $f \colon L \to L$ is the first return map, it follows from Lemma \ref{lemma cyclic} that $f$ is endperiodic with some period $n_e$.  Therefore, as above, we can find a set of junctures $J_i$ starting from a neighbourhood $U$ of $e$.  As $f$ is the first return map defined by a transverse 1-dimensional foliation $\mL_f$, flowing along $\mL_f$ each $J_i$ is carried into $J_{i+t}$ at time $t \in \bN$.  More specifically, $f$ defines a semi-map (local homeomorphism with path lifting property) $p \colon U \to V$ which takes each $B_i$ (locally) onto $V$, and which maps each $J_i$ to $J_{i+1}$. Therefore, $J_V:=p(J_i)$, for $i \geq 0$, is a well-defined 1-manifold in $V$, and is also referred to as a {\it juncture} of $\mF$ in the component $V$ of $R(\ga)$.  Note that the construction of the juncture depends on the choice of the fundamental neighbourhood system $\mU$.

Using junctures one may easily describe the behaviour of a taut depth one foliation near a compact leaf.  As always, let $V$ be a component of $R(\ga)$.  Then there exists a {\it spiral staircase (neighbourhood)} $\mN_V$ associated to $\mF$ and $\mL_f$ \cite[Sec.\,12.2.2]{Handel-Miller}.  Again suppose that $L$ is a noncompact leaf of $\mF$, and $e$ is an end of $L$ that is asymptotic to $V$.  For the sake of definiteness suppose that $e$ is an attracting end.  Using the same notation as before, suppose that $\mU$ is a fundamental neighbourhood system of $e$, and for some $i$, $J_i$ is a positive juncture for $e$.  Then there is a surface $T_i \cong J_i \times [0,1]$  transverse to $\mF$ such that for each point $x \in J_i$, the arc $x \times [0,1] \in T_i$ is contained in a flowline of $\mL_f$.  Note that $T_i \subset T_{i+1}$. The surface $B_i \cup T_i$ separates $M$ into two connected components; the component which contains $V$ is called a {\it spiral staircase neighbourhood} denoted by $\mN_V$, or denoted by $\mN_V^{\mU,i}$ if we wish to emphasise the choice of $\mU$ and $i$.  Observe that $\mN_V^{\mU, i+1} \subset \mN_V^{\mU,i}$, so the choice of $i$ can be seen as affecting the ``size'' of the neighbourhood; in other words, any given collar neighbourhood of $V$ contains $\mN_V^{\mU,i}$ for any choice of $\mU$ and for all $i>i_0$, for some $i_0$ large enough.

\subsection{Foliation cones of Cantwell and Conlon} \label{subsection foliation cones}  Let $N$ be a compact, connected $n$-manifold and let $\om \in \Om^1(N)$ be a closed, nonsingular 1-form.  Then $\om$ defines a codimension-1 foliation $\mF$ of $N$.  Thus we call $\om$ a {\it foliated form}.  Note that by Tischler's theorem \cite[Thm.\,9.4.2]{Foli1} $N$ admits such a form if and only if $N$ fibres over $S^1$.  Now suppose that $\mF$ is a depth one foliation of a sutured manifold $(M,\ga)$.  Then there is a fibration $p\colon M_0 \to S^1$.  So the foliation determines a {\it foliated class} $[\om] \in H^1(M;\bR)$, where $\om := p^*(dt) \in \Om^1(M_0)$ is the pullback of  the standard form $dt \in \Om^1(S^1)$ via the fibration $p$.    Clearly $\om$ defines a foliation $\mF_0:=\mF|M_0$.  In particular, $\om$ ``blows up nicely'' at $R(\ga)$ \cite[p.\,3.9] {CC99}, which means that $\mF_0$ can be completed to the foliation $\mF$ by adjoining the connected components of $R(\ga)$ as leaves.  

Let $\om \in H^1(M;\bR)$ and let $\mF$ be a foliation determined by $\om$.  Cantwell and Conlon define a {\it foliated ray} $[\mF]$ to be a ray $t \cdot \om$ for $t \in \bR$,  $t>0$ in $H^1(M;\bR)$ issuing from the origin \cite[Cor.\,4.4]{CC99}.   If $\mF$ is a depth zero or depth one  foliation, then the foliation ray is called a {\it proper foliated ray} \cite[p.\,36]{CC99} and $[\mF]$ is then given by  $[\mF]=\{t \cdot \la(\mF)\}_{t >0}$, since $\la(\mF)$ is precisely the foliated form determining $\mF$.  

In general, a foliated form defines foliations $\mF$ which are tangent to $R(\ga)$, have no holonomy, and their leaves are dense in $M_0$ \cite{Hector}.   Moreover, any smooth foliation with holonomy only along $R(\ga)$ is $C^0$ isotopic to a foliation defined by a foliated form \cite[Cor.\,4.4]{CC99}.  If $R(\ga) =\emptyset$, then the foliated ray determines the foliation up to smooth isotopy \cite{LB,QR}.  If a foliated form corresponds to an integral lattice point of $H^1(M;\bR)$, then it determines a depth zero or depth one foliation, up to equivalence (see Theorem \ref{thm equivalence}).  Cantwell and Conlon have shown that in all other cases the foliated ray also determines the foliation up to smooth-leaved isotopy \cite{CCOpen}.

\begin{thm} \cite[Thm.\,1.1]{CC99} \label{thm CCcones}
 Let $(M,\ga)$ be a sutured manifold.  If there are depth zero and depth one foliations $\mF$ of $(M,\ga)$, then there are finitely many open, convex, polyhedral cones in $H^1(M)$, called foliation cones, having disjoint interiors and such that the foliated rays $[\mF]$ are exactly those lying in one of these cones.  The proper foliated rays are exactly the foliated rays through points of the integral lattice and determine the corresponding foliations up to isotopy.
\end{thm}

Let  $\fC (M,\ga)$ denote the interior of the foliation cones in $H_2(M,\bdd M;\bR)$ obtained by Poincar\'e duality.

\section{Duality of the sutured Floer polytope and the foliation cones} \label{section duality}

As was remarked in Section \ref{section intro}, the moral of Lemma C is known to experts \cite{Conlon private, Gabai private}, but the author was unable to find any written references.  Cantwell and Conlon are preparing a paper exploring the relationship of sutured manifold decomposition and foliations from the perspective of staircases and junctures \cite{CCSmooth} that will include a proof of Lemma C.  Nonetheless, for the reader's convenience, we give our own proof of Lemma C, together with the minimal necessary background from foliation theory.

 Therefore, the first part of this section is dedicated to the ``classical'' topology of taut depth one foliations and the proof of Lemma C using Gabai's construction and the theory of spiral neighbourhoods given in Subsections \ref{subsection Gabai's construction} and \ref{subsection junctures}.  The second part of this section covers the proofs of Theorems A and B.  We conclude the paper by discussing some concrete examples that illustrate the duality of Theorem B.

\subsection{The classical topology of depth one foliations on sutured manifolds}

Let $(M,\ga)$ be a sutured manifold as defined by Gabai.

\begin{lemC} 
Suppose $(M,\gamma)$ is a connected sutured manifold.  Let $(M,\ga)\leadsto^S (M',\ga')$ be a surface decomposition along $S$ such that $(M',\ga')$ is taut.  Then $(M',\gamma')$ is a connected product sutured manifold
if and only if either
\begin{enumerate}
\item $R(\ga)=\emptyset$ and $S$ is the fibre of a depth zero foliation $\mF$ given by a fibration $\pi \colon M \to S^1$, \\
or
\item $R(\ga) \neq \emptyset$ and $S$ can be spun along $R(\ga)$ to be a leaf of a depth one foliation $\mF$ of $(M,\ga)$.
\end{enumerate}
Up to equivalence, all depth zero and depth one foliations of $(M,\ga)$ are obtained from a surface decomposition resulting in a connected product sutured manifold.
\end{lemC}

\begin{lemma} \label{lemma well-groomed}
Let $(M,\ga)$ be a connected sutured manifold such that $R(\ga) \neq \emptyset$ and suppose $(M,\ga) \rightsquigarrow^{S} (M',\ga')$ is a {\rm well-groomed} surface decomposition giving a connected product sutured manifold $(M',\ga')$.  Then there is a depth one foliation on $(M,\ga)$.
\end{lemma}
\begin{proof}
This lemma is just a particular case of \cite[Thm.\,5.1] {Gabai}; see subsection \ref{subsection Gabai's construction}.

\end{proof}

We now remove the well-groomed condition, thereby proving the 'only if' direction of the Lemma \nolinebreak C.

\begin{lemma} \label{lemma no conditions}
Let $(M,\ga)$ be a connected sutured manifold such that $R(\ga) \neq \emptyset$ and suppose $(M,\ga) \rightsquigarrow^{S} (M',\ga')$ is a surface decomposition giving a connected product sutured manifold $(M',\ga')$.  Then there is a depth one foliation on $(M,\ga)$.
\end{lemma}

\begin{proof}
We assume that the reader is familiar with the construction and the notation in Case 2 of subsection \ref{subsection Gabai's construction}. As before let $V$ be a component of $R_+(\ga)$.   Let $\mJ:=\{J_1, \ldots, J_k\}$ denote the connected components of $\bdd S \cap V$; the collection $\mJ$ are the junctures of $V$.  

We can apply the construction from  Case 2 of Subsection \ref{subsection Gabai's construction} at every component of $R(\ga)$ even though $S$ is not necessarily well-groomed.  This gives a foliation $\mF$ of $(M,\ga)$. To prove the lemma we need to show three things: that every leaf in the interior of $M$ is noncompact, that every such leaf is of depth one (and hence totally proper), and that $\mF$ is taut.  

Firstly, if the spinning procedure were to yield a compact leaf in the interior that would mean that $[J_{i_1}] + \cdots + [J_{i_n}] = 0$ for some subset $J_{i_1}, \ldots, J_{i_n}$ of junctures in $\mJ$.  But this is impossible as cutting along $S$ would result in a manifold $(M',\ga')$ with $\abs{ R(\ga')} >2$, which would mean that $(M',\ga')$ is not a product.

Secondly,  let $L_1$ and $L_2$ be two arbitrary noncompact (and not necessarily distinct) leaves of $\mF$.  If we can show that $L_1$ is not asymptotic to $L_2$, then each noncompact leaf is totally proper and has depth one.  To show this we use the following observation.  If $L_1$ is asymptotic to $L_2$, then for every point $x \in L_2$ and for each arc $a$ transverse to $\mF$ and passing through $x$ the set $a \cap L_1$ clusters at $x'$.

Let us label the noncompact leaves of $\mF$ according to the level set of the product foliation $\mP$ of $(M',\ga')=(\Si  \times [0,1],\bdd \Si \times [0,1])$ that gives rise to each leaf.  So for some $s,t \in(0,1)$, we have two (not necessarily distinct) leaves  $L_s'=\Si \times s$ and $L_t'=\Si \times t$.  Let $L_s$ and $L_t$ be the leaves of $\mF$ such that $L_s \supset L_s' $ and $L_t \supset L_t'$.

Let $x$ be a point in $L_t$ and let $a$ be an arc in $M$ that is transverse to $\mF$.  We need to show that $a \cap L_s$ does not cluster around $x$.  Suppose that $x$ is in $L'_t$. Then $x$ is unaffected by the spinning of $\mF$ near $R(\ga)$.  So we can choose a sufficiently small open neighbourhood $U_x$ of $x$ in $M$ such that $L_s \cap U_x = \emptyset$ if $s \neq t$, and $L_s \cap U_x = \{ x \}$ if $s=t$. Therefore, $a \cap L_s$ does not cluster around $x$.  

 Now suppose that $x \not \in L'_t$.  Then either $x \in D_i$ or $x \in K$.  Recall that $D_i = J_i \times [0,1] \times [0,\infty]$ and $K =(R(\ga) \setminus N) \times [0,\infty]$, where $N$ was a neighbourhood of $J_i$ as described in the construction of $\mF$.  In either case, $x$ is of the form $y \times (t+m)$, for some point $y$ in $J_i \times [0,1]$ or in $R(\ga) \setminus N$ and some nonnegative integer $m$.  Moreover, because of the choice of parametrisation and gluing given in \eqref{shift one}, any point of $L_s \setminus L_s'$ is of the form $z \times (s+n)$, for some nonnegative integer $n$.  If $s \neq t$, then
 \[
 \min \{ \abs{s+m-t-n} : m,n \in \bZ \}>0.
 \] 
If $s =t$, then by definition of transversality the arc $a$ can only intersect $L_t$ at points of the form $y \times (t+m)$, where the second coordinate points are discrete.  In both cases there is again a small neighbourhood $U_x$ such that $L_s \cap U_x = \emptyset$, and $L_s \cap U_x=\{x\}$, respectively.  So two noncompact leaves of $\mF$, $L_1$ and $L_2$ are not asymptotic.
 
Lastly, $\mF$ is taut for the same reason as when $S$ was well-groomed.
\end{proof}

Before we proceed to prove the `if' direction of Lemma C, in Lemma \ref{lemma Gabai's truncation}, consider the following definition.

\begin{deff} \label{deff truncating}
Let $\mF$ be a depth one foliation of $(M,\ga)$.  Suppose $L$ is a noncompact leaf of $\mF$.  Then a surface $\Si$ is obtained by {\it truncating} $L$, if $\Si =L \setminus \mN$, for some spiral staircase neighbourhood $\mN$ of $R(\ga)$.
\end{deff}

We can now summarise the content of Lemma \ref{lemma Gabai's truncation}: from a depth one foliation $\mF$ of $(M,\ga)$, one can obtain a surface $S$ giving a product decomposition by truncating an arbitrary noncompact leaf $L$ of $\mF$.  Morally, we take a leaf and remove its ends and what we are left with is the surface.  The details follow.

\begin{lemma} \label{lemma Gabai's truncation}
Suppose that $\mF$ is a depth one foliation on $(M,\ga)$.  Then there exists a surface decomposition  $(M,\ga) \rightsquigarrow^S (M',\ga')$ giving a product manifold.
\end{lemma}

\begin{proof}
Lemma \ref{lemma fibring} says that a depth one foliation $\mF$ determines a fibration $M_0 \to S^1$.  For each component $V$ of $R(\ga)$, let $\mN_V$ be a spiral staircase neighbourhood of $V$ associated to $\mF$ and to some transverse foliation $\mL$ (see Subsection \ref{subsection junctures} for details and notation).  Recall that the construction of $\mN_V$ starts by choosing a leaf $L_0$ and an end $e$ of $L_0$ converging to $V$, followed by a choice of a fundamental neighbourhood system $\mU$ of $e$, which yields a set of junctures $J_1, J_2, \ldots J_k$.  The boundary of each $\mN_V$ is $V \cup B_i \cup T_i$, for some $i \in \bN$, where $B_i \subset L_0$, and $T_i$ is obtained by flowing $J_i$ along $\mL$.   

 Set $\mN:=\cup_{V \subset R(\ga)} \mN_V$.  Then $\overline M:=M \setminus \mN$ is clearly homeomorphic to $M$ and  determines a sutured manifold $(\overline M,\overline \ga)$ with the same sutures as $M$.  Consider the effect of removing $\mN$ on the foliation $\mF$.  Let  $L$ and $L'$ be two leaves of $\mF$ transverse to $\mN$.  This implies that $\overline L:=\overline M \cap L$ is homeomorphic to $\overline {L'}:= \overline M \cap L'$ via a map defined using $\mL$.
 
 The boundary $\bdd \overline M$ is the union of two subsets $T$ and $B$, where $T$ is the subset of $\bdd \mN$ such that the leaves of $\mF| \overline M$ are transverse to $T$, and $B:= \bdd \overline M \setminus \Int T$.  Note that $B \subset \cup_{i=1}^k \overline {L_i}$, where $L_i$ are some leaves of $\mF$ and $\overline{L_i}:=L_i \cap \overline M$. Set $K:=\cup_{i=1}^k \overline {L_i}$ . As $\overline L$ is homeomorphic to  $\overline {L'}$, the manifold $M\setminus K$ is a fibration over the unit interval with fibre $\overline L$.
 
Set $S:=K \setminus B$.  Then $S$ is a decomposing surface of $(\overline M, \overline \ga)$, so we have a decomposition $(\overline M, \overline \ga) \leadsto^S (M',\ga')$ and we know that $M'  \cong \overline L \times I$.  It remains to show that $\ga' = \bdd  \overline L \times I$.  Choose the orientation of $S$ to be opposite from the orientation of $K$.  By the definition of a surface decomposition
\[
\ga':=(\ga \cap M') \cup (S_+ \cap R_-(\overline \ga)) \cup (S_- \cup R_+(\overline \ga)).
\]
Further, by the definition of $\mF$ all of the fibres are transverse to $\ga$.  A careful consideration of orientations now shows that $(S_+ \cap R_-(\overline \ga)) \cup (S_- \cup R_+(\overline \ga))=B$, possibly after a small isotopy of the boundary.  Thus, $\ga'$ is precisely the subset of $\bdd M'$ to which the fibres of the fibration are transverse; in other words, $\ga'=\bdd L' \times I$.

It follows that $(M,\ga) \leadsto^S (M',\ga')$ yields a product sutured manifold $(M',\ga')$.  
\end{proof}

\begin{rmk} \label{rmk leaf homology}
Note that if $\al$ is a loop in $M$, then $\langle \la(\mF), [\al] \rangle$ is the signed intersection number of $\al$ with a noncompact leaf $L$.  Truncating $L$ by a sufficiently small spiral staircase neighbourhood in Lemma \ref{lemma Gabai's truncation}, we have that $\langle PD \circ [S],[\al] \rangle=\langle \la(\mF),[\al]\rangle$ for any loop $\al$.  Hence,  $\la(\mF)= PD \circ[S]$ where $S$ is the decomposing surface obtained by truncating $L$.  
\end{rmk}

The following is a corollary of Remark \ref{rmk leaf homology} and of Theorem \ref{thm equivalence}.

\begin{cor} \label{cor equivalent foliations}
Let $\mF$ and $\mF'$ be two depth one foliations of $(M,\ga)$, together with the decomposing surfaces from Lemma \ref{lemma Gabai's truncation}, $S$ and $S'$, respectively.  Then $\mF$ is equivalent to $\mF'$ if and only if $[S]=[S'] \in H_2(M,\bdd M)$.
\end{cor}

\begin{proof} [Proof of Lemma C]
When $R(\ga)= \emptyset$ and $(M',\ga')$ is a product, then $\mF$ is a fibration.  Conversely, when $\mF$ is a fibration, cutting along a fibre gives a product.  For the depth one case Lemmas \ref{lemma no conditions} and \ref{lemma Gabai's truncation} each prove one direction of the theorem.   
\end{proof}

\begin{rmk} \label{rmk restricting to balanced}
Recall that in Remark \ref{rmk depth zero and foliations} we distinguished two cases when $M$ fibres over $S^1$.  The first case is when $R(\ga) =\emptyset$ and the fibration is a depth zero foliation (think of $S^1 \times D^2$ fibered with $T(\ga)=S^1 \times \bdd D^2$).  The second case is when $A(\ga) \neq \emptyset$, and the fibres are not transverse to $\ga$ or not tangential to $R(\ga)$ (think $S^1 \times D^2$ with two parallel sutures).  In both cases cutting along a fibre can result in a connected product sutured manifold.  From the proof of Lemma C, it is now evident that in the first case our construction from the surface decomposition recovers the depth zero foliation (i.e. the fibration), but that in the second case we construct a depth one foliation. 

Note that in Lemma C, we do not restrict to balanced sutured manifolds, therefore $R(\ga) = \emptyset$ can happen.  However, in order to work with the sutured Floer polytope in the following subsection, we must restrict to strongly balanced sutured manifolds, so only the latter case of fibring can occur.
\end{rmk}

\subsection{The duality}

Let $(M,\ga)$ be a strongly balanced sutured manifold.  Recall that given a decomposing surface $S$, there exists a set of outer $\Spin^c$ structures denoted by $O_S$; see Definition \ref{def outer}.

\begin{deff} \label{deff extremal}
Let $(M,\ga)$ be a taut, strongly balanced sutured manifold.  
\begin{enumerate}
\item A $\Spin^c$ structure $\fs$ is called {\it extremal} if there exists a surface $S$ such that  $\{\fs\} =O_S \cap S(M,\ga)$.  Then $\fs$ is {\it extremal with respect to $\al:=[S]$}, and this is equivalent to saying that $\al(\fs) > \al(\ft)$ for any other $\ft \in S(M,\ga)$ (see Theorems \ref{thm nice} and \ref{thm decomposition}).
\item  The polytope $P(M,\ga)$ is said to have an {\it extremal $\bZ$ at $\fs$}, if $\fs$ is extremal  and \linebreak $SFH(M,\ga,\fs)=\bZ$.  
\end{enumerate}
\end{deff}

The following lemma is a direct consequence of Juh\'asz's work; it is useful for us to write it in the terminology from Definition \ref{deff extremal}.

\begin{lemma} \label{lemma decomposition}
Let $(M,\ga)$ be a taut, strongly balanced sutured manifold with $H_2(M)=0$.  Then $\fs$ is an extremal $\bZ$ with respect to a homology class $\al$ if and only if there exists a surface decomposition $(M,\ga) \leadsto^S (M',\ga')$ such that $(M',\ga')$ is a product and $\al=[S]$.
\end{lemma}

\begin{proof}
Let $\fs$ carry a $\bZ$ extremal to $\al$.  Then by Theorem \ref{thm decomposition}, there exists a surface decomposition $(M,\ga) \leadsto^S (M',\ga')$ such that $(M',\ga')$ is taut, $\al=[S]$ and $SFH(M',\ga')=SFH(M,\ga,\fs)=\bZ$.  Theorem \ref{thm product} implies that $(M',\ga')$ is a product.  Conversely, if such a decomposition is given, then
\[
\bigoplus_{\fs \in O_S} SFH(M,\ga, \fs)=SFH_\al(M,\ga)=SFH(M',\ga')=\bZ,
\]
where the second equality comes from Theorem \ref{thm nice} and the third from Theorem \ref{thm product}.  The result follows.
\end{proof}

Lemma C and Lemma \ref{lemma decomposition} lead to Theorem A.

 \begin{thmA}
Suppose $(M,\ga)$ is a strongly balanced sutured manifold with $H_2(M;\bZ)=0$, and let $P(M,\ga)$ denote its sutured polytope. Then $P(M,\ga)$ has an extremal $\bZ$ at a $\Spin^c$ structure $\fs$ if and only if there exists a taut depth one foliation $\mF$ of $(M,\ga)$ whose sole compact leaves are the connected components of $R(\ga)$ and such that $\fs$ is extremal with respect to $PD \circ \la(\mF)$.
\end{thmA}

\begin{proof}
Let $\fs$ be an extremal $\bZ$ of $P(M,\ga)$.  By Lemma \ref{lemma decomposition}, there exists a decomposing surface $S$ with $(M,\ga) \leadsto^S (M',\ga')$ such that $[S]=\al$ and $(M',\ga')$ is a product.  By Lemma C, we can construct $\mF$, and by Remark \ref{rmk leaf homology} we have that $\la(\mF)=PD \circ [S]$.
Conversely, given such a foliation $\mF$, by the proof of Lemma C, truncating a noncompact leaf $L$ gives a decomposing surface $S$ with $\la(\mF)= PD \circ [S]$.  Since we get a product $(M',\ga')$ when decomposing along $S$, applying Lemma \ref{lemma decomposition} completes the proof.
\end{proof}

The disadvantage of the sutured Floer polytope is that it is well defined only up to translation in $H^2(M,\bdd M;\bR)$.  Thus, a choice needs to be made for there to exist a well-defined {\it dual polytope}.  However, without making any choices we can define the {\it dual sutured cones} $Q(M,\ga)$ in $H_2(M,\bdd M; \bR)$. 

Let $P$ be a polytope given by vertices $v_1, \ldots, v_{n}$ living in a vector space $V$ over some field $\bF$.  Then, as we said in the Introduction, the dual cones $Q$ can be defined to be a collection of polyhedral cones $Q_1, \ldots, Q_n$ in the dual space $V^*= \Hom(V, \bF)$ where
\[
Q_i:=\{v^* \in V^* : v^*(v_i) > v^*(v_j) \textrm{ for } i \neq j\}.
\]

\begin{deff} \label{deff dual cones} 
Define the {\it dual sutured cones $Q(M,\ga)$} in $H_2(M,\bdd M ;\bR)$ to be the dual cones of the sutured polytope $P(M,\ga)$, with each labeled by the corresponding extremal $\Spin^c$ of $P(M,\ga)$.  The cones that correspond to extremal $\bZ$ vertices of $P(M,\ga)$ are called the {\it extremal $\bZ$ cones} and are denoted by $Q_\bZ(M,\ga)$.
\end{deff}

In the introduction we mentioned how one defines a dual cone of any polytope in an affine space given just by its vertices.  Definition \ref{deff dual cones} just repeats this definition only in the language that is most useful here.

\begin{lemma} \label{lemma same as CCcones}
Let $(M,\ga)$ be a taut, strongly balanced sutured manifold with $H_2(M)=0$. Then the closure of each subset $C_\fs$ of the dual sutured cones $Q(M,\ga)$ is indeed a convex, polyhedral cone.    Also, if $\dim P(M,\ga)=b_1(M)$, then the closure of $Q(M,\ga)$ covers all of $H_2(M,\bdd M;\bR)$. 
\end{lemma}

\begin{proof}
This is just stating that $Q(M,\ga)$ is defined as a dual to $P(M,\ga)$.
\end{proof}

\begin{thmB} \label{thm C}
Let $(M,\ga)$ be a taut, strongly balanced sutured manifold with $H_2(M)=0$.  The extremal $\bZ$ cones of $Q(M,\ga)$ are precisely the foliation cones $\fC(M,\ga)$ defined by Cantwell and Conlon in \cite{CC99} (see Theorem \ref{thm CCcones}).
\end{thmB}

\begin{proof}
By Theorem \ref{thm CCcones} the foliation cones $\fC(M,\ga)$ are open, convex, polyhedral cones in \linebreak $H_2(M,\bdd M;\bR)$.  Further, the integral homology classes $\al \in H_2(M,\bdd M)$ in $\fC(M,\ga)$ are precisely those for which there exists a foliation $\mF$ of depth one such that $\la(\mF)=PD \circ \al$.  By Theorem \nolinebreak A, it follows extremal $\bZ$ points of $P(M,\ga)$  correspond precisely to such foliations.  In particular, $Q_\bZ$ is a collection of open, convex, polyhedral cones whose integral homology classes correspond to  depth one foliations via the same correspondence of \linebreak $\al=PD \circ \la(\mF)$.  So $\fC(M,\ga)$ is the same as $Q_\bZ(M,\ga)$. 

\end{proof} 

We conclude with a corollary that describes the sutured manifold analogue of the Thurston norm and its fibred faces for closed 3-manifolds. 

Denote by $B_y$ the the polytope in $H_2(M,\bdd M;\bR)$ that is the unit ball  of the Juh\'asz's seminorm $y$ described in Remark \ref{rmk Ju norm}.  The faces of $B_y$ that are dual to extremal $\bZ$ $\Spin^c$ structures in $-P(M,\ga)$ are called the {\it foliated faces}.  

\begin{cor}
Let $(M,\ga)$ be a taut, strongly balanced sutured manifold with $H_2(M)=0$.  Then each foliation cone of $\fC(M,\ga)$ is subtended by a foliated face of $B_y$.
\end{cor}

As $y$ involves the somewhat artificial choice of putting the centre of mass of $P(M,\ga)$ at $0 \in H^2(M,\bdd M;\bR)$, the obvious question is why is this corollary not phrased in terms of $y_t$.  Of course, a similar statement could be made for the the unit ball of the geometric sutured function $y_t$, however only if it makes sense to talk about the unit ball, that is, if $y_t$ is at least a seminorm.

\subsection{Examples}  \label{subsection examples}

 Finally, we illustrate Theorem B by checking that the examples of foliations cones computed by Cantwell and Conlon \cite{CC99,CC sutured Thurston norm} are indeed dual to the associated sutured Floer polytopes.  

Let $L$ be a knot or link in $S^3$, and $R$ a minimal genus Seifert surface of $L$.  Then denote by $S^3(R)$ the strongly balanced sutured manifold obtained by removing an open neighbourhood of $R$ from $S^3$, that is, $S^3(R):=\left( S^3 \setminus \Int(R \times I), \bdd R \times I\right)$. 

Let $P(2r,2s,2t)$ denote the the standard three-component pretzel link, and let $R$ be the Seifert surface obtained by the Seifert algorithm.   Examples 2 and 5 from \cite{CC99} describe the foliation cones of $S^3(R)$ for $P(2,2,2)$ and $P(2,4,2)$, respectively.  The sutured Floer polytopes for these examples were computed in \cite[Ex.\,8.6]{FJR10}, and it is not hard to see that they are indeed dual to the foliation cones.

Example 4 in \cite{CC99} describes the foliation cones of $S^3(R)$ for a 2-component link and the Seifert surface given in Figure \ref{fig CClink}.

\begin{figure}[h]
\centering
\includegraphics [scale=0.5]{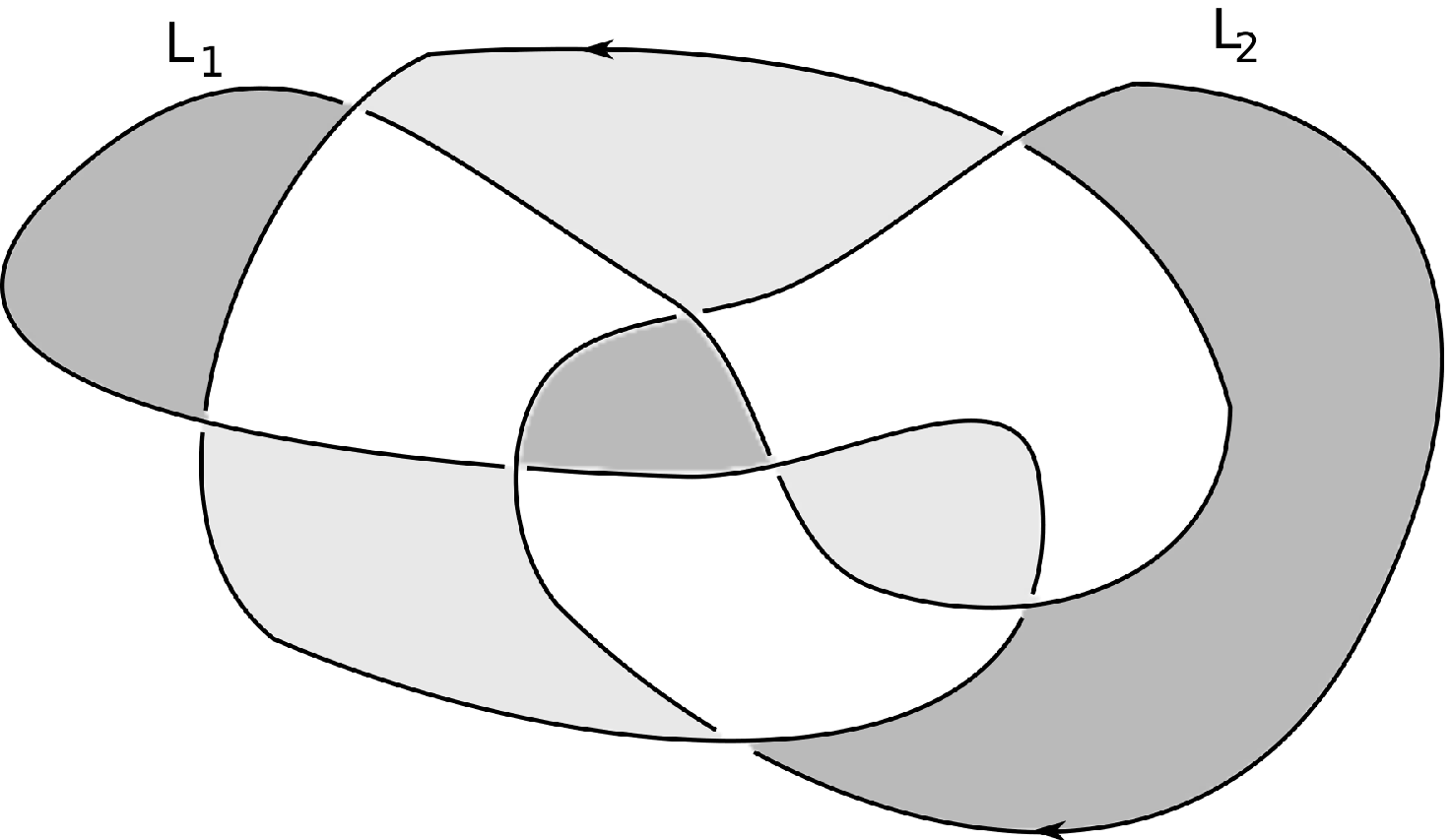}
\hspace{1cm}
\includegraphics [scale=0.5]{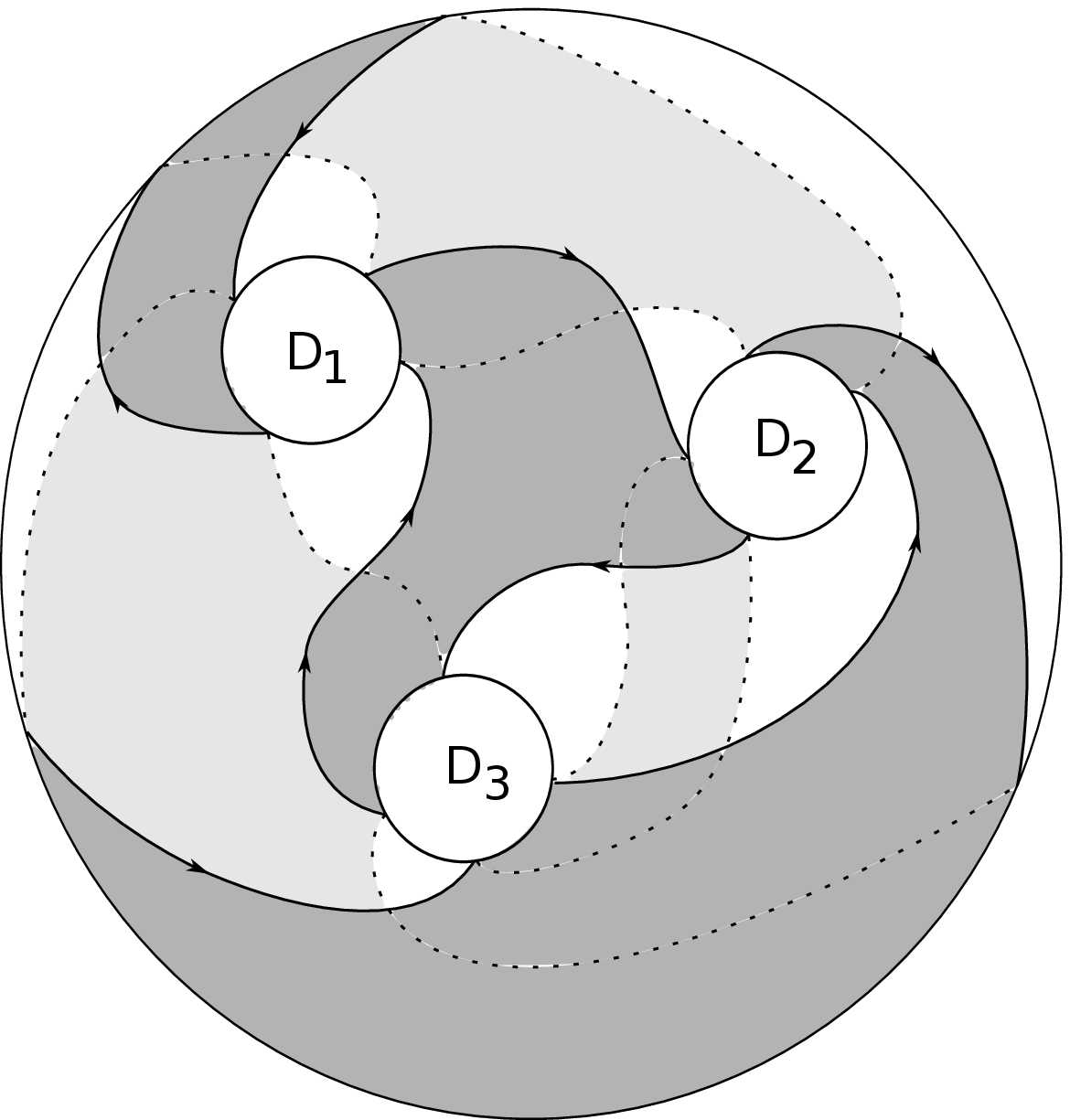}
\caption{The 2-component link and its Seifert surface.}
\label{fig CClink}
\end{figure}

As this is a non-split alternating link, by \cite[Cor.\,6.11]{FJR10} it follows that $S^3(R)$ is a {\it sutured L-space}, that is, the group $SFH(M,\ga,\fs)$ is either trivial or isomorphic to $\bZ$ for every $\Spin^c$ structure $\fs$ \cite[Cor\,6.6]{FJR10}.  Therefore, $SFH(M,\ga)$ and, in particular $P(M,\ga)$, can easily be computed from the map $\pi_1(R_-(\ga)) \to \pi_1(M)$ using Fox calculus \cite[Prop.\,1.2]{FJR10}.  

\begin{convention}
For the remainder of the paper, homology groups are understood to be taken with coefficients in $\bR$.
\end{convention}

First, let us describe the foliation cones. Considering the isomorphism given by Poincar\'e duality $H^1(M) =H_2(M,\bdd M)$, we take the foliation cones to live in $H_2(M,\bdd M)=\bR^3$.  Following Example 4 and Figure 14 from \cite{CC99}, denote by $e_1,e_2,e_3$ the basis of $H_2(M,\bdd M)$, given by the disks $D_i$ in Figure \ref{fig CClink}.  Let $e_0:=-(e_1+e_2+e_3)$.  
Then there are five convex foliation cones whose closures cover all of $H_2(M,\bdd M)$: four of the cones are 3-sided, and the fifth is 4-sided.  The cones are determined by rays through five points 
\[
-e_2-e_3, e_2 , -e_0 , e_3 , -e_1.
\]
In Figure \ref{fig FoliationConeEx} these points have been connected in such a way that the sides of the pyramid subtend each cone.  This is the easiest way for visualising the duality with the sutured Floer polytope.

\begin{figure}[h]
\centering
\includegraphics [scale=0.6]{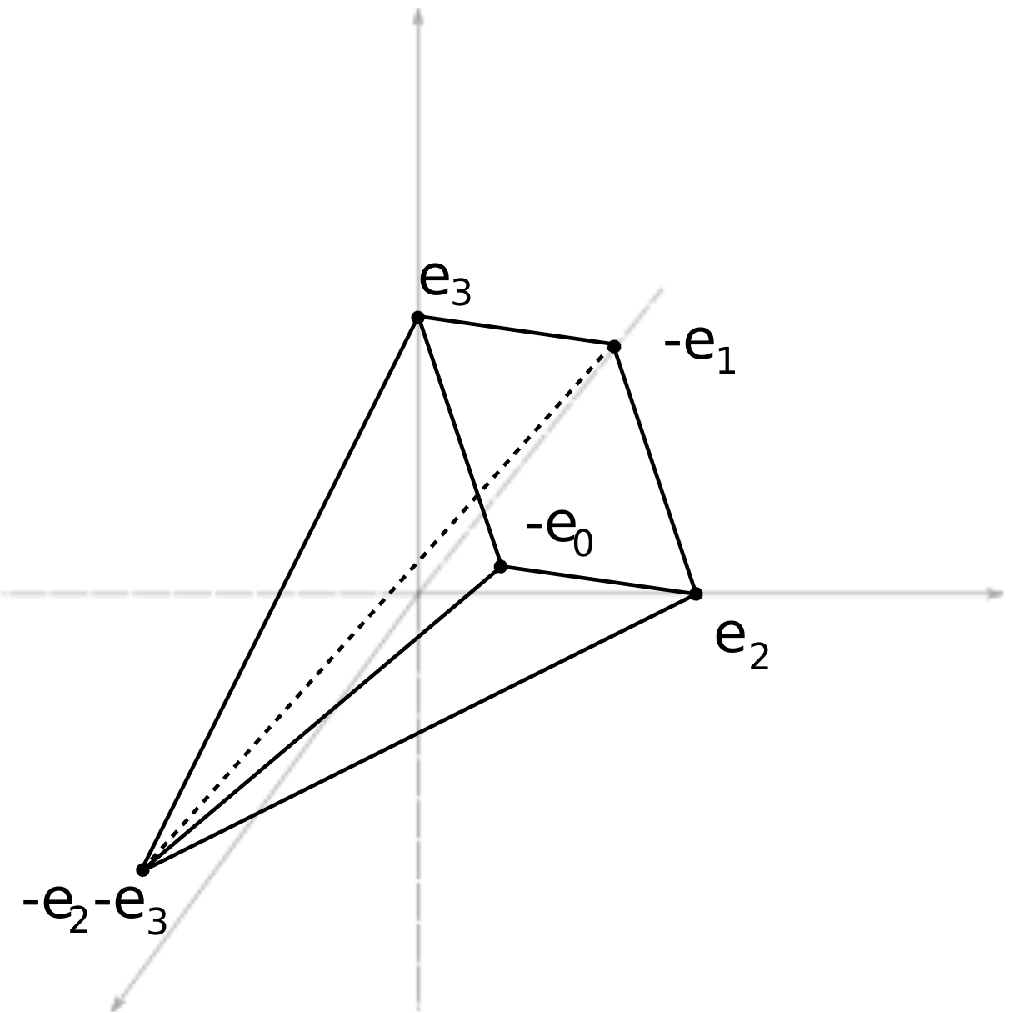}
\hspace{1cm}
\includegraphics [scale=0.6]{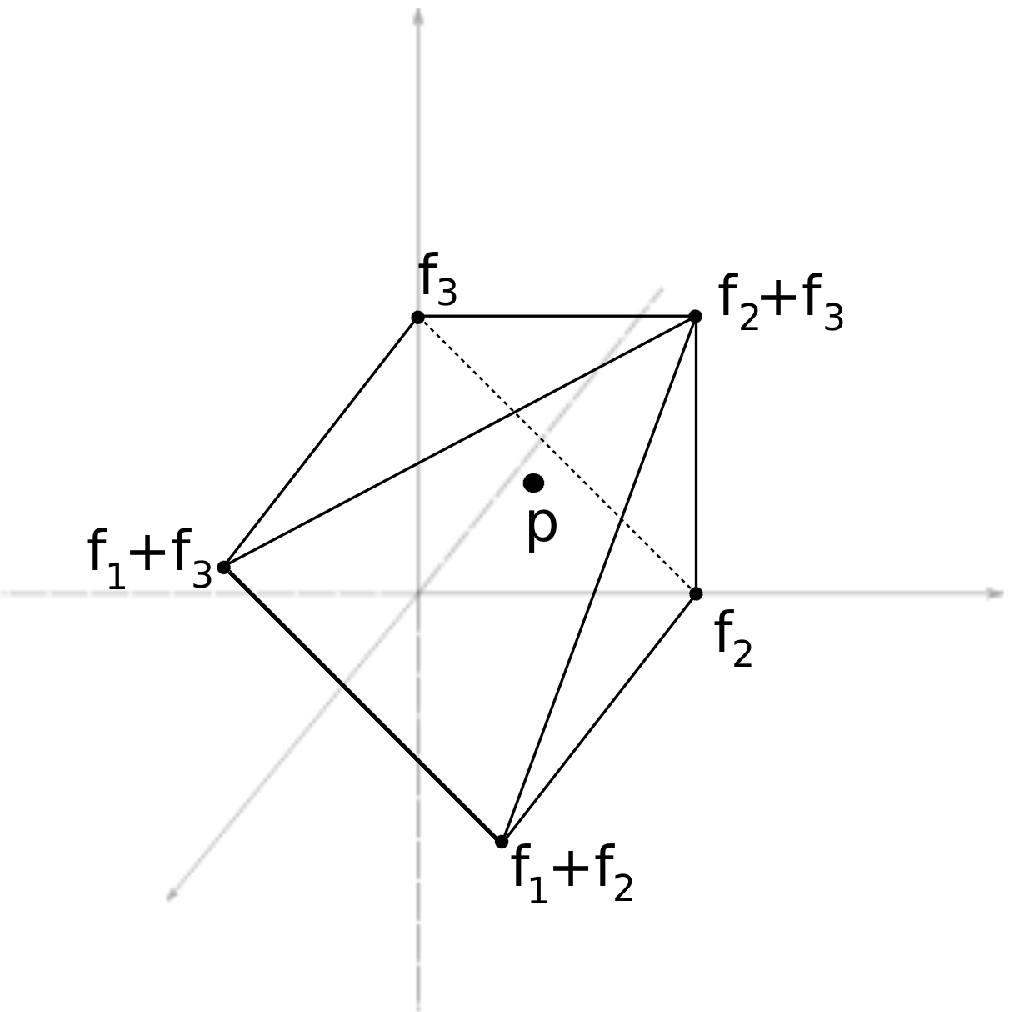}
\caption{Left: the foliation cones given by rays through the vertices of the pyramid in $H_2(M,\bdd M)=\bR^3$.  Right: the sutured Floer polytope given up to translation in $H^2(M,\bdd M)=\bR^3$, with $p$ the centre of mass.  }
\label{fig FoliationConeEx}
\end{figure}

Let $f_1,f_2,f_3$ be the basis of $H^2(M,\bdd M)$ dual to $e_1,e_2,e_3$; here``dual'' refers to the duality of $H^2(M,\bdd M)$ and $H_2(M,\bdd M)$ as vector spaces.  Using the same notation, and the method of computation via Fox calculus, we find that the sutured Floer polytope $P(S^3(R))$ in $H^2(M,\bdd M)$ is a pyramid with a rectangular basis given by the vectors (up to translation)
\[
f_2+f_3,f_2,f_3,f_1+f_3,f_1+f_2,
\]
where $f_2+f_3$ is the apex of the pyramid in Figure \ref{fig FoliationConeEx} .

To compute the dual of $P(S^3(R))$, we first have to find the centre of mass $p$ of the polytope, then translate the polytope so that the centre of mass is at $0 \in H^2(M,\bdd M)$.  A bit of elementary geometry shows that $p=\frac{1}{5}(2f_1+3f_2+3f_3)$ in the current coordinate system.  Translate the polytope, or equivalently change the coordinates, so that $p=0$.  The dual cones of $P(S^3(R))$ are then given by five rays normal to the five sides of $P(S^3(R))$.  Using symmetries of the polytope it is not hard to compute that these rays  precisely pass through $-e_2-e_3, e_2, -e_0, e_3,-e_1$, which described the foliation cones.

\begin{rmk}
In all three of the above examples, $x^s=z$, that is, the sutured Thurston norm of $(M,\bdd M)$ agrees with the symmetrised sutured seminorm $z(\al)=\frac{1}{2}\left(y(\al) + y(-\al)\right)$.  This equality does not hold in general, as was shown in \cite[Prop.\,7.16]{FJR10} using an example of Cantwell and Conlon \cite[Ex.\,2]{CC sutured Thurston norm}.  From their respective computations it is not hard to check that the sutured Floer polytope and the foliation cones are dual.
\end{rmk}

\begin{rmk}
In \cite{Altman1} we show that there exists an infinite family of knots and pairs of Seifert surfaces $R_1$ and $R_2$ associated to each knot, where the polytopes $P(S^3(R_1))$ and $P(S^3(R_2))$ are not affine isomorphic.  In other words, the sutured Floer polytope of $S^3(R)$ is not a knot invariant, and therefore, by Theorem B it is also true that the foliation cones of a Seifert surface complement are not a knot invariant. 
\end{rmk}


\vspace{1cm}

{\sc \noindent Mathematics Institute, Zeeman Building, University of Warwick, UK.}
\\
E-mail address: irida.altman@gmail.com

\end{document}